%% file: main_arkiv.tex
\documentclass{article}

\usepackage[utf8]{inputenc} 
\usepackage[T1]{fontenc}    
\usepackage{hyperref}       
\usepackage{url}            
\usepackage{booktabs}       
\usepackage{nicefrac}       
\usepackage{microtype}      
\usepackage{graphicx}
\usepackage{subfigure}
\usepackage{caption}
\usepackage{graphics}
\usepackage[space]{grffile}
\usepackage{xcolor}         
\usepackage{amssymb,amsmath,amscd,amsfonts,amsthm,bbm,mathrsfs,yhmath}
\usepackage{mathtools}
\usepackage{caption}
\usepackage{array}
\usepackage{framed}

\usepackage{multirow}

\usepackage{bbm}

\usepackage{mathrsfs}

\usepackage{aliascnt}
\usepackage{cleveref}
\usepackage{titlesec}
\usepackage{paralist}
\usepackage{xargs}
\usepackage[disable]{todonotes}
\usepackage{pdfpages}
\usepackage{comment}
\usepackage[shortlabels]{enumitem}
\usepackage{algorithm}

\theoremstyle{definition}
\newtheorem{theorem}{Theorem}
\newtheorem{lemma}[theorem]{Lemma}

\newtheorem{prop}[theorem]{Proposition}
\newtheorem{definition}{Definition}
\newtheorem{rem}{Remark}

\newcommand{\subscript}[2]{$#1 _ #2$}
\newlist{assumplist}{enumerate}{1}
\setlist[assumplist]{label=(\subscript{\textbf{A}}{{\arabic*}}),resume=count_cond}
\Crefname{assumplisti}{Assumption}{Assumptions}

\renewcommand\labelenumi{(\roman{enumi})}
\renewcommand\theenumi\labelenumi

\input{defs.tex}

\title{Adaptive Importance Sampling meets Mirror Descent: a Bias-variance tradeoff}

\author{Anna Korba\footnote{CREST, ENSAE, Institut Polytechnique de Paris}\ \  and François Portier\footnote{CREST, ENSAI, address: \texttt{francois.portier@gmail.com}}}

\begin{document}
	
	\maketitle
	
%
%
%

	\begin{abstract}

Adaptive importance sampling is a widely spread Monte Carlo technique that uses a re-weighting strategy to iteratively estimate the so-called target distribution. A major drawback of adaptive importance sampling is the large variance of the weights which is known to badly impact the accuracy of the estimates. This paper investigates a regularization strategy whose basic principle is to raise the importance weights at a certain power. This regularization parameter, that might evolve between zero and one during the algorithm, is shown (i) to balance between the bias and the variance and (ii) to be connected to the mirror descent framework. Using a kernel density estimate to build the sampling policy, the uniform convergence is established under mild conditions. Finally, several practical ways to choose the regularization parameter are discussed and the benefits of the proposed approach are illustrated empirically.
	\end{abstract}

	\input{intro.tex}

\input{scheme.tex}

\input{results.tex}

	\input{ada.tex}

	\input{related_work}

	\input{numerical.tex}
	\input{conclu.tex}

\paragraph{Acknowledgment.} The authors would like to express their gratitude to Kamelia Daudel for her many valuable remarks and interesting suggestions.

\bibliography{main_arkiv.bbl}
\bibliographystyle{unsrt}
\newpage

	\appendix

\input{appendix.tex}

\end{document}

%% file: defs.tex
\newcommand{\ak}[1]{\todo[color=green!40,size=\footnotesize]{ \textbf{AK:}  #1}}

\newcommand{\arginf}{\mathrm{arginf}}

\newcommand{\argmin}{\mathrm{argmin}}

\newcommand{\data}{\mathcal{D}}
\newcommand{\eqdef}{:=}
\newcommand{\eqsp}{\;}
\DeclareMathOperator{\KL}{KL}
\newcommandx{\kl}[2][1=q, 2=f]{\KL\lr{#1 \|#2}}
\newcommand{\lr}[1]{\left(#1 \right)}
\newcommand{\lrb}[1]{\left[#1 \right]}

\newcommand{\lrcb}[1]{\left\{#1 \right\}}
\newcommand{\nset}{\mathbb{N}}
\newcommand{\nstar}{\mathbb{N}^\star}
\newcommand{\PP}{\mathbb{P}}
\newcommand{\PQ}{\mathbb{Q}}
\newcommand{\tq}{\tilde q}
\newcommand{\sq}{q^{\star}}
\newcommand{\rmd}{\mathrm d}
\newcommand{\rset}{\mathbb{R}}
\newcommand{\Rset}{\mathbb{R}}

\newcommand{\Q}{\mathcal{Q}}

\newcommand{\Var}{\mathbb{V}\mathrm{ar}}

\newcommandx{\wEta}[1][1=\eta_k]{W_{n,k}^{(#1)}} 

%% file: intro.tex
\section{Introduction}

Many machine learning problems rely on our ability to compute the expected value of a function $g$ of interest according to a target distribution $f$. Unfortunately, these calculations may be hindered by the fact that we are unable to sample from $f$, which is often intractable. 
Adaptive Importance Sampling (AIS) is one increasingly popular way to tackle this problem in the machine learning literature.  
It is notably used in stochastic optimization to achieve variance reduction and speed-up of stochastic gradient descent \cite{el2020adaptive,needell2014stochastic,zhao2015stochastic,bouchard2015accelerating,stich2017safe,katharopoulos2018not,johnson2018training,csiba2018importance}. 
In addition, it is commonly used in reinforcement learning \cite{huang2020importance, liu2020understanding} and has also been integrated recently to the tuning of bayesian neural networks \cite{ariafar2021faster}.


The idea of AIS is to sample from an alternative, simpler \textit{proposal} probability density $q_k$ at time $k$ of the algorithm to approximate $f$.
It updates the proposal distribution at each iteration based on samples from the previous iteration and their associated importance sampling weights, in an online manner.  AIS methods in the literature differ on the type or number of proposals, on the weighting scheme, and on the update of the proposal. Most methods use parametric proposals, e.g. Gaussians \cite{martino2015adaptive} or generic parametric proposals \cite{cornuet2012adaptive}, and update the parameters of these proposals at each iteration based on weighted moment estimation of the target distribution. Alternatively, nonparametric methods were proposed using kernel density estimates of the target \cite{west1993approximating,givens1996local,zhang1996nonparametric,neddermeyer2009computationally}.

It is well-known that a bad proposal can lead to a high variance in high dimension \cite{bengtsson2008curse} 
and part of the art of using AIS then consists in selecting a \textit{policy} $(q_k)_{k \geq 1}$, i.e. a sequence of proposals, that is practical and provides good efficiency gains, especially as the dimension increases \cite{ho+b:1992,owen+z:2000,douc+g+m+r:2007b,hartmann2021nonasymptotic, portier2018asymptotic} (see \cite{bugallo2017adaptive,elvira2021advances} for detailed reviews). 

In this paper, we propose a new non parametric Adaptive Importance Sampling method, that (i) introduces a new regularization strategy which raises adaptively the importance sampling weights to a certain power ranging from $0$ to $1$ (ii) uses a mixture between a kernel density estimate of the target and a safe reference density as proposal. In particular, (i) will be connected 
to the entropic mirror descent algorithm on the space of probability distributions as in \cite{dai2016provable}
, while (ii) will be closely related to the safe adaptive importance sampling approach of \cite{delyon2021safe}. We prove that this algorithm enjoys a uniform convergence result and we exhibit a way to carry out the regularisation procedure adaptively. On practical AIS applications, this procedure allows us
to make significant speed-ups and performance gains empirically compared to the (classical) AIS framework.



The paper is organized as follows. In \Cref{sec:moti}, we motivate our regularisation strategy by showing how it balances between the bias and the variance in an importance sampling setting. Furthermore, we underline its connections with an entropic mirror descent algorithm over the space of probability distributions. \Cref{sec:rais} is devoted to presenting our Safe and Regularised Adaptive Importance Sampling (SRAIS) approach. Notably, we state our theoretical guarantees and we describe our regularization schedule at each iteration. Lastly, our
numerical results are to be found in \Cref{sec:numerical}.

%% file: scheme.tex
In the whole paper, all densities are with respect to the Lebesgue measure and $\int g $ denotes the integral of $g$ with respect to the Lebesgue measure.

\section{Motivation}
\label{sec:moti}




Given a target density function $f$ defined on $\mathbb R^d$, our goal is to compute $\int gf$ for some (possibly many) integrable functions $g$. 
%
%
In many situations it is impossible or too complicated to get samples from $f$ and use a classical Monte Carlo estimator; in which case we resort to Importance Sampling (IS) methods.

Letting $q$ be another density function defined on $\mathbb R^d$ and assuming that $X$ is a random variable with distribution $q$, the basic idea of IS is to re-weight $g(X)$ by $W(X) = f(X)/q(X)$, the so-called \textit{importance weight}. Whenever $q$ dominates $f$, using the fact that $\mathbb E [W(X)g(X)] = \int g  f$  and samples $X_1,\dots,X_n\sim q$ , one can build 
an (unbiased) IS estimator of $ \int g  f$ as
\begin{equation}\label{eq:IS}
	\int gf	\approx \frac{1}{n}\sum_{k=1}^n \frac{f(X_k)}{q(X_k)}g(X_k) =  \frac{1}{n}\sum_{k=1}^n W(X_k) g(X_k) .
\end{equation}

\begin{rem}\label{rk:snis}
	In many applications, e.g. in Bayesian inference, $f$ is often only known up to a normalization constant. In this case, one can use the normalized counterpart of \eqref{eq:IS}, namely the self-normalized importance sampling (SNIS) estimator defined as $\sum_{k=1}^n W(X_k) g(X_k)/\sum_{k=1}^n W(X_k) $,  where $X_k \sim q$; and which suffers from a small bias of order $\mathcal{O}(n)$ \cite{owen2013monte}.
\end{rem}
However, the main issue observed in practice is that a large variability of the importance weights results in a high variance of the IS estimators \cite{bengtsson2008curse}. This is certainly the case when $q$ is not sufficiently close to $f$. To tackle this problem, we introduce the concept of \textit{regularized importance sampling}, which we describe now.

Regularized importance sampling consists in using \textit{regularized} weights of the form $W(X)^\eta$, $\eta\in (0,1)$ instead of working with the classical importance weight $W(X)$. In doing so and as written in  \Cref{lem:bias_variance} below, we will be able to diminish the variance at the cost of increasing the bias, which illustrates the typical bias-variance tradeoff and justifies the use of the term regularized importance sampling.
\begin{lemma}\label{lem:bias_variance}
 Suppose that $q$ dominates $f$ and define $W(X) = f(X)/q(X) $ with $X$ having density $q$. For all $\eta\in (0,1]$, it holds that
$\mathbb{E} [W(X)^\eta] \leq 1 
\text{ and }
\Var [W(X)^\eta] \leq \Var[W(X)]$.
If $q\neq f$, then both inequalities become equalities if and only if $\eta = 1$.
\end{lemma}
The proof of this lemma is deferred to \Cref{sec:proof_lemma_bias_variance} and we now make a few comments.
A first comment is that via a proper choice of the regularization parameter $\eta$, one will be able to balance between the bias and the variance in IS. A second comment is that if we consider the estimate $W(X)^\eta g(X)$ of $\int g  f $ for some regularization parameter $\eta \in (0,1)$, it has the following expectation
$$\mathbb E [W(X)^\eta g(X)] = \int f ^\eta q  ^{1-\eta}  g$$
when $X$ is generated from the density $q$.  
This suggests that regularized importance sampling moves from an initial density $q$ to $f ^\eta q  ^{1-\eta} $ (up to a normalization constant) while the classical IS case $\eta = 1$ directly targets $f$ starting from $q$. 

\begin{rem}
	The latter regularization differs from the one considered in simulated annealing (or tempering) \cite{zhou2016toward}, where $f^{\eta}$  interpolates between $q$ the uniform distribution to $f$, as (\textit{the temperature}) $\eta$  increases from $0$ to $1$. In this setting, the importance weights would be proportional to $f^{\eta}/q$ instead of $(f/q)^{\eta}$.
\end{rem}

The regularization of IS weights we consider actually relates to a well-known scheme 
 called \textit{entropic mirror descent} \cite{beck2003mirror}.
This scheme has been widely used in the machine learning literature, e.g. in variational inference
\cite{dai2016provable,daudel2020infinite}, reinforcement learning \cite{orabona2019modern} (sometimes referred to continuous exponential weights),  optimal transport \cite{mishchenko2019sinkhorn} or to train generative models \cite{hsieh2019finding}.
Given an initial probability density function $q_1$, the sequence of proposals $(\sq_k)_{k \geq 1}$ obtained by entropic mirror descent for Kullback-Leibler minimisation satisfies:
\begin{align} \label{eq:emd}
\sq_{k+1} \propto f^{\eta_k}  (\sq_k)^ {1-\eta_k} \eqsp, \quad k \geq 1 \eqsp,
\end{align}
where $(\eta_k)_{k \geq 1} $ is a sequence of positive learning rates (we refer to \Cref{sec:NoteEMD} for a derivation of this iterative formula from 
an optimisation perspective). When $(\eta_k)_{k \geq 1}$ is valued in $(0,1]$, which will be assumed throughout the paper, the iterative scheme defined by \eqref{eq:emd} is attractive as it satisfies the following lemma, whose proof can be found in \Cref{sec:proof_lemma_tv}.

\begin{lemma}\label{lemma:tv}
Let $(\eta_k)_{k \geq 1}$ valued in $(0,1]$ and $(\sq_k)_{k\geq 1}$ be defined by \eqref{eq:emd} starting from an initial probability density function $q_1$. Then, for all $n \in \nstar$,
\begin{align*}
\int | f - \sq_{n+1} | \leq \sqrt{ 2 \kl[f][q_1] }  \prod_{k=1}^{n}  ( 1-\eta_k )^{1/2 },
\end{align*}
where $\kl[f][q_1] = \int \log\lr{f/ q_1} f$ is the Kullback-Leibler divergence between $f$ and $q_1$.
\end{lemma}



\begin{rem} When $\kl[f][q_1] $ is finite,
\Cref{lemma:tv} yields different convergence rates depending on the regularization schedule. Taking $\eta_k = c  / k$ with $0< c< 1$ yields an $O(   n^{-c / 2}   ) $ convergence rate, while taking $\eta_k = c / {k}^\beta$ with $0< c< 1$ and $\beta\in [0,1)$ yields an $O(   \exp( - C n^{(1-\beta) }   )     )  $  convergence rate, for some $C>0$ (details are given in \Cref{sec:proof_lemma_tv} right after the proof of \Cref{lemma:tv}).
Note that we recover the case of the constant learning policy when $\beta = 0$
. This is in contrast with convergence results from the optimisation 
 literature, as recalled in \Cref{lemma:cv:emd} of \Cref{sec:NoteEMD}. In the latter, convergence is established under stronger assumptions on 
the iterates $q_k$ (namely, log boundedness of the ratios $\sq_k/f$)
 and on the learning policy (e.g, $\eta_k = c/\sqrt{k}$ , i.e. $\beta=1/2$
, which yields a $O(\log(n)/\sqrt{n})$ convergence rate). 
\end{rem}

Unfortunately, running iteration \eqref{eq:emd} is not feasible for two reasons: (i) each iteration depends on the whole function $f$ (not some evaluations of $f$), which is unknown; (ii) even when $f$ is known, computing $\sq_{k+1}(x)$ can be intractable due to the normalization following \eqref{eq:emd} to ensure $\sq_{k+1}$ is a probability distribution density .


\section{Safe and Regularized Adaptive Importance Sampling}\label{sec:rais}
This section introduces Safe and Regularized Adaptive Importance Sampling and relates it to the entropic mirror descent algorithm presented in the previous section.

\subsection{Proposed scheme}
We propose an \textit{Adaptive Importance Sampling} (AIS) method, which in contrast with classical Importance Sampling \eqref{eq:IS} which uses a unique proposal distribution $q$, uses a  \textit{policy}, i.e. a sequence of proposals $(q_k)_{k\ge0}$ in an online manner. We define this notion rigourously below.
\begin{definition}
Let $(X_k)_{k\geq 1}$ a  sequence of random variables defined on $(\Omega, \mathcal F,\mathbb P)$, and define $\mathcal F_{k} =\sigma(X_1,\ldots, X_{k})$ for $k\geq 1$ and $\mathcal F_0 = \emptyset$. The sequence $(X_k)_{k\geq 1}$  is said to have policy $(q_k)_{k\geq 0}$ whenever for all $k \geq 1$,
$X_ k \sim q_{k-1}$,
conditionally to $\mathcal F_{k-1} $.
\end{definition}
However, in AIS, one should ensure that the proposal $q_k$ sufficiently covers the support of the target $f$ at each iteration $k\ge 0$.
Among recent approaches,  \textit{safe adaptive importance sampling} (SAIS) \cite{delyon2021safe} proposes to 
choose the policy $(q_k)_{k\geq 0}$ as a mixture between a certain kernel density estimate $f_k$ of $f$ and some “safe” probability density $q_0$ with heavy tails compared to the ones of $f$. Letting $(\lambda_k)_{k\geq 1}$ valued in $[0,1]$ be a decreasing sequence, we seek $(q_k)_{k\geq 0}$ under the form
\begin{align}\label{eq:sampler_update}
q_{k}  = (1-\lambda_k)  f_k + \lambda_k q_0 \eqsp, \qquad \forall k\geq 1 \eqsp.
\end{align}
In this way, the mixture with $q_0$ shall prevent too small values for $q_k$ and a too high variance on the importance weights \cite{delyon2021safe}. In contrast, the other part in the mixture, $f_k$, is meant to speed-up the convergence of $q_k$ to $f$.

Let us now define the kernel density estimate $f_k$ of $f$ more precisely, which will differ from the one used in \cite{delyon2021safe}. For this purpose, let $K:\mathbb{R}^d \to \mathbb \Rset^{+}$ be a density called \textit{kernel} and for a \textit{bandwith} $h\ge 0$, define $K_h(u) = K(u/h)/h^d$. Let $(h_k)_{k\geq 1}$ valued in $\Rset^{+}$ be a sequence of 
\textit{bandwidths}.
The kernel estimate of $f$ at step $n$ denoted $f_n$, is defined  by
\begin{equation}\label{eq:kernelEstimate}
f_n (x) = \sum_{k=1}^n \wEta K_{h_n} (x - X_k) \eqsp ,\qquad\forall x\in \mathbb R^d,
\end{equation}
where $\forall k=1,\ldots, n,$
\begin{equation}
\wEta \propto W_k^{\eta_k} =\left(\frac{f (X_k) }{ q_{k-1}(X_k)}\right)^{\eta_k}
  \text{ s.t. }  \sum_{k=1}^n \wEta = 1.\label{eq:normalised_weights}
\end{equation}
In other words, $f_n$ is based on the kernel evaluated at points $(X_k)_{k=1}^n$, where each point $X_k$ is weighted by  $\wEta$ for $k=1,\dots,n$. The weights  $(\wEta)_{k=1}^n$ correspond to the \textit{normalized} counterpart of the \textit{regularized} (i.e., raised to the power $\eta_k$) importance weights $(W_k)_{k=1}^n$.
The algorithm resulting from the previous description can be written as follows.

%
%
%
%
%
%

\begin{algorithm}
\caption{\em Safe and Regularized Adaptive Importance sampling (SRAIS)} \label{algo:RIS}
\textbf{Inputs}: The safe density $q_0$, the sequences of  bandwidths $(h_k)_{k = 1,\ldots, n}$, mixture weights $(\lambda_k)_{k = 1,\ldots, n}$, learning rates $(\eta_k)_{k=1, \ldots, n}$. \smallskip

For $k= 0,1,\ldots, n-1 $:
 \begin{enumerate}
   \item Generate $X_{k+1}\sim q_k$.
   \item Compute (a) $W_{k+1} = f(X_{k+1} ) / q_k(X_{k+1})$ and (b) $(W_{k+1, j}^{(\eta_{j})})_{1 \leq j \leq k+1}$.
   \item Return $q_{k+1}=(1-\lambda_{k+1})f_{k+1}+\lambda_{k+1} q_0$ where $f_{k+1}=\sum_{j=1}^{k+1} W_{k+1, j}^{(\eta_{j})}K_{h_{k+1}}(\cdot -X_{j})$.
 \end{enumerate}
\end{algorithm}
The algorithm consists in three steps: (1) sample a new particle from the current proposal, (2) compute the new importance weight and its normalized, regularized counterpart, (3) update the kernel density estimate of the target, and the proposal as the mixture of this  estimate and the safe density.

In 
\Cref{algo:RIS}, we recognize at each iteration $k$ a stochastic approximation of the mirror descent scheme \eqref{eq:emd}, as the expectation of $W_k^{\eta_k} K_{h_k} (x -X_k)$ 
w.r.t. $q_k$  is equal to $(f^{\eta_k} q_{k-1}^{1 -\eta_k}  \star K_{h_k} ) (x)  $, where $\star$ denotes the convolution operator.  From classical convolution results \cite{tsybakov2008introduction}, the latter quantity approximates $f^{\eta_k} q_{k-1}^{1 -\eta_k}  $ when the bandwith $h_k$ is small.
Notice also that in \Cref{algo:RIS}, at each step $k$, the bandwith $h_k$ is shared across all the generated points $(X_j)_{j=1}^{k+1}$. However, one could choose a different bandwith for each point, as this does not have an impact on our theoretical results (see \Cref{rk:bandwiths} in the Appendix).


%% file: results.tex
\subsection{Uniform convergence of the scheme}\label{sec:results}

We now turn to the theoretical analysis of \Cref{algo:RIS}. Let us first introduce,  for all $n \geq 1$, and for any $x\in \mathbb R^d$,
\begin{align*}
	N_{n} (x) =\frac{1}{n}\sum_{k=1} ^{n} W_{k} ^{\eta_k} K_{h_n} (x  - X_{k}), \;
	D_{n} = \frac{1}{n}\sum_{k=1} ^{n} W_{k} ^{\eta_k}.
\end{align*}
Using this notation, the sequence $(f_n)_{n\geq 1}$ from \eqref{eq:kernelEstimate} can be rewritten under the form $f_{n} (x)= N_n(x)/D_n$, for all $n \geq 1, \, x\in \mathbb R^d$. Here, the analysis is carried out by examining separately $(D_n)_{n\geq 1}$ and $(N_n)_{n\geq 1}$. In particular the function $N_{n} $, which is an unnormalized smooth density estimate of $f$, can be decomposed as follows:
\begin{align}\label{eq:nn_moins_f}
N_n  -  f  = M_n + \left\{ \left(\frac{1}{n}\sum_{k=1} ^n (f^{\eta_k} q_{k-1}^{1-\eta_k} - f)\right)  \star K_{h_n} \right\}    
+ \big\{ f\star  K_{h_n}   -f \big\} \eqsp,
\end{align}
where the function $M_n$ is defined by, for all $x\in \mathbb R^d$,
 \begin{align}\label{eq:mn}
M_n(x) =\frac{1}{n}\sum_{k=1}^n \left\{ W_k^{\eta_k} K_{h_n} (x  - X_{k}) 
\right.
\left.
 -  \{f^{\eta_k} q_{k-1}^{1-\eta_k}\} \star K_{h_n} (x)\right\} .
 \end{align}
 The previous decomposition \eqref{eq:nn_moins_f} sheds some light on the approach taken in the mathematical proof. Indeed, the two last terms on the right hand side of this decomposition represent two different types of bias.  The first one with $f^{\eta_k} q_{k-1}^{1-\eta_k} - f$ captures the bias induced by the regularization (it is therefore directly related to the mirror descent update) and should be negligible when $\eta_n\to 1$ as $n\to \infty$. The second one with $f\star  K_{h_n} -f$ is well studied in nonparametric statistics \cite{tsybakov2008introduction} and goes to $0$ as soon as $h_n\to 0$ when $n\to\infty$. Finally, the first term $M_n$ in \eqref{eq:nn_moins_f}  is an average of martingale increments (see \Cref{prop:mg} in the Appendix), 
which
will help our theoretical analysis 
via Freedman-type concentration inequalities \cite{freedman1975tail,bercu2015concentration}.

Let us move on to our theoretical results and  introduce the following assumptions:
\begin{assumplist}[leftmargin=0.4in]
	\item  \label{cond:sequences} \begin{enumerate}[leftmargin=0.13in,label=(\roman*)]
				\item  The sequence $(\lambda_k)_{k\geq 1}$ is valued in $ (0,1]$, nonincreasing, and  $\lim_{k\to \infty}\lambda_k= 0$ and $\lim_{k\to \infty }\log(k) / (k\lambda_k)= 0$.
		\item The sequence $(h_k)_{k\geq 1}$ is valued in $\Rset^{+}$, nonincreasing, and  $\lim_{k\to \infty}h_k = 0$ and $\lim_{k\to \infty}\log(k) / (kh_k^d\lambda_k)= 0$.
				\item  The sequence $(\eta_k)_{k \geq 1}$ is valued in $ (0, 1]$, and $\lim_{k\to \infty} \eta_k=1$, $\lim_{k\to \infty}( 1-\eta_k ) \log(h_k)=0$
		and $\lim_{k\to \infty}(1-\eta_k) \log(\lambda_{k-1}) =0$.
	\end{enumerate}

	
	\item \label{cond:lowerbound} The density $q_0$ is bounded and there exists $c>0$ such that for all $x\in \mathbb R^d$, $q_0(x) \geq c f(x)$.
	
	\item \label{cond:reg_f} The function $f$ is nonnegative, $L$-Lipschitz and bounded by $U\in \Rset^{+}$.
	
	\item  \label{cond:kernel} $\int K = 1$, $\int \|u\| K(u) \rmd u < \infty$, $\int K^{1/2} <\infty$ and $\int \|u\| K(u)^{1/2} \rmd u < \infty$. The kernel $K$ is bounded by $K_\infty\ge 0$ and is $L_K$-Lipschitz with $L_K>0$, i.e. :
	\begin{align*}
		|K(x+ u ) -K(x)| \leq L_K \|u\| \quad \text{for all } x,u\in \Rset^d.
	\end{align*}

\end{assumplist}
In \ref{cond:sequences}, the reasonably slow convergence of $(\lambda_k)_{k\ge1}$ to 0 helps the proposal to visit a large part of the space, while the faster convergence of $(h_k)_{k\ge1}$ and $(\eta_k)_{k\ge1}$ to $0$ and $1$ respectively enables to remove the bias asymptotically. 
Then, \ref{cond:lowerbound} permits to control the variance induced by the weights whatever the proposal $q_k$. In practice, since $q_0$ is chosen by the user, this assumption is relatively weak. For instance,  \ref{cond:lowerbound} is valid when $f$ has exponential decay and $q_0$ is any member of the Student's $t$-distributions family (polynomial decay).  Finally, the regularity of $f$ and $K$ required in \ref{cond:reg_f}-\ref{cond:kernel} is used to ensure the convergence of $(f\star K_{h_k}) $ toward $f$.  \ref{cond:reg_f} is valid for many families of distributions (e.g., Student's, Gaussian, beta, exponential and any mixture of these densities) and  \ref{cond:kernel} is satisfied by most kernels 
 (Gaussian, beta, Epanechnikov,...), see \cite{tsybakov2008introduction}.

Under these assumptions, we obtain first the following uniform convergence result for the estimate \eqref{eq:kernelEstimate}, whose proof is given in \Cref{sec:proof_initial_bound}.


\begin{prop}\label{prop:initial_bound}
Assume \ref{cond:sequences}, \ref{cond:lowerbound}, \ref{cond:reg_f} and \ref{cond:kernel}. Then, for any $r > 0$, we have that almost-surely
\begin{align*}
\sup_{\|x\|\leq n^r } | f_n(x) - f(x) | \to0 \quad \text{as } n\to \infty.
\end{align*}
\end{prop}
The previous result is now extended to a result over the whole space. If we strenghten our assumptions on $q_0$ and $K$ as follows, we can obtain a stronger uniform convergence result.

\begin{assumplist}[leftmargin=0.4in]
\item \label{cond:tail_f} There exist $(c, \delta) \in (\Rset^{+})^2$, such that for all $x\in \mathbb R^d$,
$c (1+\|x\|^\delta ) f(x)  \leq  q_0(x) .$
\item \label{cond:tail}
There exist $(C_K,r_K)\in  (\Rset^{+})^2$, such that for all $x\in \mathbb R^d$,
$K(x)  \leq C_K (1+\|x\|) ^{-r_K}$.
\end{assumplist}

Assumptions \ref{cond:tail_f} and \ref{cond:tail} are not too restrictive, as the examples of $f$ and $q_0$ as well as the kernel functions given before are still valid under this new framework. Because \ref{cond:tail_f} is stronger than \ref{cond:lowerbound}, the latter will no longer be needed in our second result, whose proof is given in \Cref{sec:proof_initial_bound2}.
\begin{prop}\label{prop:initial_bound2}
Assume \ref{cond:sequences}, \ref{cond:reg_f}, \ref{cond:kernel}, \ref{cond:tail_f} and \ref{cond:tail}. Then, we have that almost-surely
\begin{align*}
\sup_{x \in \mathbb R^d } | f_n(x) - f(x) | \to 0 \quad \text{as } n\to \infty.
\end{align*}
\end{prop}



%
%

\begin{rem}\label{rk:batch}
	With some additional notation, \Cref{algo:RIS} could be extended to a mini-batch version where a fixed number of particles $m_k\ge 1$ is generated at each iteration instead of $m_k=1$. Such a choice does not impact our proof since the martingale property we use in the proofs of convergence remains valid. This encouraged us to use such a batch of particles to construct the sequence $(\eta_k)_{k\ge 1}$ in an \textit{adaptive} manner, as explained in the following
	\Cref{sec:ada}.
\end{rem} 

%% file: ada.tex
\subsection{Adaptive choice of the regularization parameter}\label{sec:ada}

We have explicited in \Cref{sec:results} the conditions on the hyperparameters of \Cref{algo:RIS}, so that the latter consistently estimates integrals of interest with respect to the target $f$. In particular, it requires that the sequence of regularization parameters $(\eta_k)_{k\ge 1}$ converges to 1. Here, we propose a practical and adaptive way to construct such a sequence. It relies on the following idea: when the proposal $q_k$ is equal to $f$, the importance weights are uniformly distributed on the points sampled in \Cref{algo:RIS}. The quality of approximation of the proposal w.r.t. to the target can thus be reformulated as quantifying how far is the distribution of the importance weights from the uniform distribution.
Intuitively, penalizing the divergence between the distribution reweighted by the importance weight, with respect to the uniform distribution, penalizes a high variance for the importance weights.
We choose to consider Renyi's $\alpha$-divergence \cite{renyi1961measures}, as it will enable us to influence how fast the regularization evolves.


At each time $k$, we draw a batch of $m_k$ i.i.d samples $X_{k,1}, \ldots, X_{k, m_k}$ generated from $q_{k-1}$, and compute their associated normalised IS weights $W_{k, 1}, \ldots, W_{k,m_k}$, i.e. $W_{k,\ell}  \propto f(X_{k, \ell}) / q_{k - 1} (X_{k, \ell}) $ for all $\ell = 1 \ldots m_k$ such that $\sum_{\ell =1}^{m_k} W_{k,\ell} = 1$. Notice they are not regularized as in \eqref{eq:normalised_weights}. Now, let 
$\PP = \sum_{l=1}^{m_k}W_{k,l}\delta_{X_{k,l}}$
 and 
 $\PQ= \sum_{l=1}^{m_k} \nicefrac{1}{m_k} \delta_{X_{k,l}}$ the reweighted and uniform distribution on the $(X_{k,l})_{l=1}^{m_k}$ respectively. Renyi's $\alpha$-divergence \cite{renyi1961measures} of $\PP$ from $\PQ$
is defined for $\alpha \in \rset \setminus \lrcb{0,1}$ by
$$
D_\alpha (\PP|| \PQ) = \frac{1}{\alpha -1} \log \lr{ \sum_{\ell = 1}^{m_k} W_{k,\ell}^{\alpha} m_k^{\alpha-1}}.$$
It can be extended by continuity to the forward Kullback-Leibler by letting $\alpha \to 1$ (we use the notation $D_1 (\PP || \PQ)$) and to the reverse Kullback-Leibler by letting $\alpha \to 0$ (we use the notation $D_0 (\PP || \PQ)$). Notably, for all $\alpha \in [0,1]$ (see \cite[Theorem 3]{2012arXiv1206.2459V}),  it holds that
$0 \leq D_0 (\PP || \PQ) \leq D_\alpha (\PP || \PQ) \leq D_1 (\PP || \PQ).
$
At time $k$, for  $\alpha \in \rset \setminus \lrcb{0,1}$, we propose to set the regularisation parameter as
\begin{equation}\label{eq:ada}
\eta_{k, \alpha} \eqdef 1 - \frac{D_\alpha (\PP || \PQ)}{\log(m_k)} \eqsp.
\end{equation}
Choosing the regularization parameter as \eqref{eq:ada} is particularly convenient, since it does not require to use a pre-defined sequence $(\eta_k)_{k\ge 1}$ as written in \Cref{algo:RIS}. In contrast, the regularization is chosen adaptively at each iteration, depending on how far is the current estimate to $f$. 
The following proposition, whose proof can be found in \Cref{sec:proof_tuning_eta} guarantees that our proposal \eqref{eq:ada} is a valid candidate.

\begin{prop}\label{prop:tuning_eta} Let $\alpha \in [0,1]$ and let $(\eta_{k,\alpha})_{k\geq 1}$ be the sequence defined by \eqref{eq:ada} for all $k \geq 1$. Then, we have:
	\vspace*{-0.2cm}
	\begin{enumerate}[label = (\roman*)]
		\item The sequence $(\eta_{k,\alpha})_{k\ge 1}$  is valued in $[0,1]$, with $\eta_{k,\alpha}=1$ iff $\PP=\PQ$; \label{item:ada1}
		\item $
		0 \leq \eta_{k,1} \leq \eta_{k, \alpha} \leq 1$; \label{item: ada2}
		\item Further assume that $(q_k)_{k \geq 1}$ is a sequence of probability density functions s.t. $\lim_{k\to \infty} | q_k(x)  - f(x)| = 0 $ almost everywhere and that $\lim_{k\to \infty} m_k=m$.  Then, $\lim_{k\to \infty}   \eta_{k,\alpha}  = 1$ in $L_1$ .\label{item:ada3}
	\end{enumerate}
\end{prop}

By \Cref{prop:tuning_eta}(ii), one can possibly increase the value $\eta_k$ by decreasing $\alpha$.\ak{example of f satisfying assump of Prop 5- (iii)?}
The procedure resulting from the previous description can be written as follows.

\begin{algorithm}
	\caption{\em Renyi's Adaptive Regularization (RAR)} \label{algo:ADA}
	\textbf{Input}: iteration $k$, number of samples $m$, $\alpha \in [0,1]$. \\
	\smallskip
	For $l= 0,1,\ldots, m $:
	\begin{enumerate}[leftmargin=0.4in]
		\item Generate $X_{k,1},\dots, X_{k,m}$ from $q_{k-1}$ (from \Cref{algo:RIS}).
		\item Compute the normalised importance weights $W_{k,l} \propto f(X_{k,l} ) / q_{k-1}(X_{k,l})$ such that $\sum_{l=1}^{m} W_{k,l} = 1$.
		\item Return  $\eta_{k,\alpha} = 1 - D_\alpha (\PP || \PQ)/\log(m)$, where $\PP = \sum_{l=1}^{m}W_{k,l}\delta_{X_{k,l}}$
		and 
		$\PQ= \sum_{l=1}^{m} \nicefrac{1}{m} \delta_{X_{k,l}}$.
	\end{enumerate}
\end{algorithm}
	

%% file: related_work.tex
\section{Related work}
\textbf{Particle Mirror Descent \cite{dai2016provable}}. 
Closely related to our scheme \Cref{sec:rais} is the one
proposed in \cite{dai2016provable}, that also approximates entropic mirror descent. 
Their proposal is to approximate the (intractable) $\sq_{k+1}$ of \eqref{eq:emd} by a weighted kernel density estimator centered at $m_k$ particles sampled from $q_k$ the previous approximate of $\sq_k$, and update the weights as \eqref{eq:normalised_weights}.
Our algorithm is different in nature from the one proposed in \cite{dai2016provable}, which require that both the number of iterations $k$ and number of particles at each iteration $m_k$ should grow to infinity. First, we only need one new particle at each iteration, i.e. $m_k=1$ (but the theoretical guarantees hold even if $m_k\ge 1$, see \Cref{rk:batch}), which means that our approach requires a single request to $f$ at each iteration; and we use all particles generated at times $k'\le k$. This is a major improvement making our approach in line with the spirit of \textit{stochastic approximation}  \cite{kushner2003stochastic} where each iteration should be of small cost. 
Second, our proposal is a mixture of $f_k$, that approximates \eqref{eq:emd}, and a safe density $q_0$ which allows in practice to visit exhaustively the underlying domain.
Finally, our theoretical results also differ from the ones of \cite[Theorem 6]{dai2016provable}. In our setting, the sequence $(\eta_k)_{k\ge1}$ is not intended to go to $0$ as in \cite{dai2016provable}, but to $1$ in order to realize the bias-variance trade-off described in Section \ref{sec:moti}. 
Also in contrast with the results of \cite[Theorem 6]{dai2016provable}, our results of \Cref{prop:initial_bound} and \ref{prop:initial_bound2} are free from any restriction on the behavior of the sequence $(q_k)_{k\ge 1}$, as we only impose conditions on the safe density \ref{cond:lowerbound}  and  \ref{cond:tail_f}. Their results assume that $q_k$ is bounded away from $0$, 
which is unlikely to be satisfied in practice.

\textbf{Safe Adaptive Importance Sampling \cite{delyon2021safe}.} 
 The Safe Adaptive Importance Sampling algorithm of \cite{delyon2021safe} corresponds to the particular case where the sequence $(\eta_k)_{k\ge1} $ is constant and equal to 1 in \Cref{algo:RIS}. In  \cite{delyon2021safe}, the authors obtain uniform convergence results  when the sequence $(\lambda_k)_{k\ge0}$ satisfies \ref{cond:sequences}(i). Our algorithm extends SAIS that we recover as a particular case, hence our theoretical results in \Cref{sec:results} improves and extends convergence results of \cite{delyon2021safe}.  Our adaptive choice for the regularization \Cref{algo:ADA} is also new.  Moreover, our approach outperforms SAIS ($\eta_k=1$) numerically, as will be illustrated in \Cref{sec:numerical}. 

\textbf{Regularization of IS weights.}	Non linear transformation of the weights have been promoted in the literature of IS previously. 
For instance, truncated IS (or non linear IS) \cite{ionides2008truncated} \cite{koblents2015population}, clips the weights $W$ as $W'=\min(W, \tau)$ where $\tau$ is some threshold. In \cite{gramacy2010importance}, it is proposed to elevate IS weights to some regularization parameter lambda, similarly to our work. This power lambda is set maximizing the Effective Sample Size (see eq. 6 and Prop 2.1 therein) instead of Renyi’s alpha divergence between the distribution of the importance weights and the uniform. Also these methods are proposed in the naive (i.e., non adaptive) importance sampling framework. 
In \cite{finke2019importance} (see appendix A.3) is proposed a similar strategy to \cite{gramacy2010importance} but in the AIS setting.  It differs at least in two ways from our approach. First, the theoretical guarantees of such a regularization schedule are not investigated in their paper while we establish the convergence of $\eta$ to $1$ in \Cref{prop:tuning_eta} so that our requirements for uniform convergence are satisfied. Second, in the previous work, the regularization parameter $\eta$ is chosen again to maximize the ESS, by running, at each iteration, a bisection algorithm.  It appears more expensive than the evaluation of the closed form given in \Cref{algo:RIS} which is only of complexity $O(m)$ where $m$ is the size of the batch.

\textbf{Sequential Monte Carlo (SMC).}
The algorithm proposed in this paper is formally not part of the SMC framework but still some interesting links might be drawn. SMC are often viewed as performing adaptive importance sampling whilst also allowing for the target distribution to change over iterations. The targets sequence might be chosen as the entropic mirror descent iteration as recommended for instance in \cite{del2006sequential}, eq. 7. The resulting SMC algorithm might benefit from the same advantage (small variance) as our approach.
However, the particles in traditional SMC \cite{chopin2004central,del2006sequential} are moved in a completely different manners as in the proposed AIS algorithm. When running an SMC algorithm, each particle is necessarily moved around its ``parent'' (a random walk move called mutation) and then, among the new particles obtained, some are removed based on their weights values (a subsampling step called selection). 
The presence of this selection mechanism in SMC suggests some similarities with the popular random walk Metropolis in which a particle not profitable would be removed with high probability. Our algorithm, whose particles distribution is only $(q_n)_{n\geq 1} $, is simpler due to the absence of selection. This explains why we obtain different results based on different proofs techniques than the ones used in the  SMC literature. While SMC theoretical guarantees have been expressed as central limit theorems \cite{chopin2004central}, 
we have obtained almost sure convergence of the sequence of proposals.

%% file: numerical.tex
\section{Numerical Experiments}\label{sec:numerical}

In this section, we discuss the performance of \Cref{algo:RIS} along with the subroutine \Cref{algo:ADA} to approximate integrals of interest w.r.t. a target distribution $f$, on toy  and real-world examples. Since renormalized weights are used in \Cref{algo:RIS}-\ref{algo:ADA}, these experiments only require to know the target distribution up to its normalization constant. A gaussian kernel is used in all experiments.

\subsection{Toy examples}

\paragraph{Gaussian mixtures.}   Let us now describe three interesting examples starting by the simplest and finishing with the most difficult case. Denote by $\phi_\Sigma $ the multivariate Gaussian density with mean zero and variance $\Sigma$. Denote also $\boldsymbol{I}_d$ the identity matrix and $\boldsymbol{u}_d$ the $d$-dimensional vector whose coordinates are all equal to $1$.
The \textit{Cold Start} case is when the target density is a Gaussian given by
$f_1(x) =  \phi_{\Sigma} (x - 5 \boldsymbol{u}_d / \sqrt d )$ , 
 $x\in \mathbb R^d$;
where $\Sigma = (.4 / \sqrt d )^2 \boldsymbol{I}_d  $  and the starting density is the multivariate Student distribution with mean zero and variance $(5/d) \boldsymbol{I}_d $. Note the division by $\sqrt d$ which preserves the Euclidean distance between the two centers of the starting distribution and the target, whatever the dimension. The same explanation justifies the normalization of the variances by $1/d$. The \textit{Gaussian mixture} example corresponds to the target density
$f_2(x) = .5 \phi_\Sigma (x - \boldsymbol{u}_d / (2\sqrt d) ) + .5 \phi_\Sigma (x + \boldsymbol{u}_d / (2\sqrt d)) $, 
$x\in \mathbb R^d$.
The initial density is the multivariate student distribution with mean $(1,-1, 0,\ldots, 0) / \sqrt d  $ and variance $(5/d) \boldsymbol{I}_d $. The initial mean value differs from zero to prevent the naive algorithm using the initial density from having good results (due to the symmetry). The \textit{Anisotropic Gaussian Mixture} case is similar to the previous example, except that it is unbalanced and that each Gaussian in the mixture is anisotropic. The target is given by
$f_3(x) = .25 \phi_V (x - \boldsymbol{u}_d / (2\sqrt d) ) + .75 \phi_V(x + \boldsymbol{u}_d / (2\sqrt d))$, 
$x\in \mathbb R^d$;
with $V = (.4 / \sqrt d )^2  \mathrm{diag} (10,1,\ldots, 1) $.  The starting density is the multivariate student distribution with mean $(1,-1, 0,\ldots, 0) / \sqrt d  $ and variance $(5/d) \boldsymbol{I}_d $.

\paragraph{Competitive methods.}
We focus first on four constant values for $\eta$ for \Cref{algo:RIS}, respectively $1, .75, .5, .25$. While $\eta = 1$ is an unbiased estimate with large variance, $\eta= .25$ is biased but has smaller variance.  We also run \Cref{algo:RIS} when $(\eta_k)_{k\geq 1}$ is tuned via \Cref{algo:ADA} (RAR),  with parameter $\alpha = .5$.
All algorithms are initialised by sampling independently $4.10^4$ particles from the initial density. We further sample independently $N = 18.10^3$ particles from $q_k$ at each time $k \geq 1$, for a total computational budget of $4.10^5$ particles (i.e $n = 20$ iterations). We follow the recommendation given in \cite{delyon2021safe} which consists in running a subsampling procedure of size $\ell_k = \sqrt {N_k}$ when $N_k $ denotes the total amount of particles at step $k$. These particles are used to construct the kernel estimate $f_k$, improving significantly the computing time. We take $\lambda_k \propto \ell_k ^{-2 / (4+d)}$ and $h_k \propto \ell_k ^{-1 / (4+d)}$ as recommended in our theoretical study.\ak{pas clair par rapport à avant qu'on a recommandé ça... ou alors pour l = 1 ok on dit " to satisfy \ref{cond:sequences}(i)-(ii)"}

\paragraph{Error evaluation.}
For each method in competition, the evaluation of the mean squared error (MSE) is made by computing the average of $\| \mu_f- \hat \mu\|_2^2$ over $50$ runs of the method where $\mu_f$ is the mean of the target and $\hat \mu$ stands for the estimated mean. Note that the examples where chosen in a way that the error computing the mean reflects the general behavior of the method.\ak{i don't understand this sentence} In each case, we plot the error at each iteration of the algorithm. This is presented in  \Cref{fig:toy_error,fig:toy_bp} in dimension $d = 16$ (similar results are have been obtained for $d=4$ and $d=8$ but are deferred \Cref{sec:add_expes}).


\begin{figure}[h]
\centering
		\includegraphics[scale=0.27]{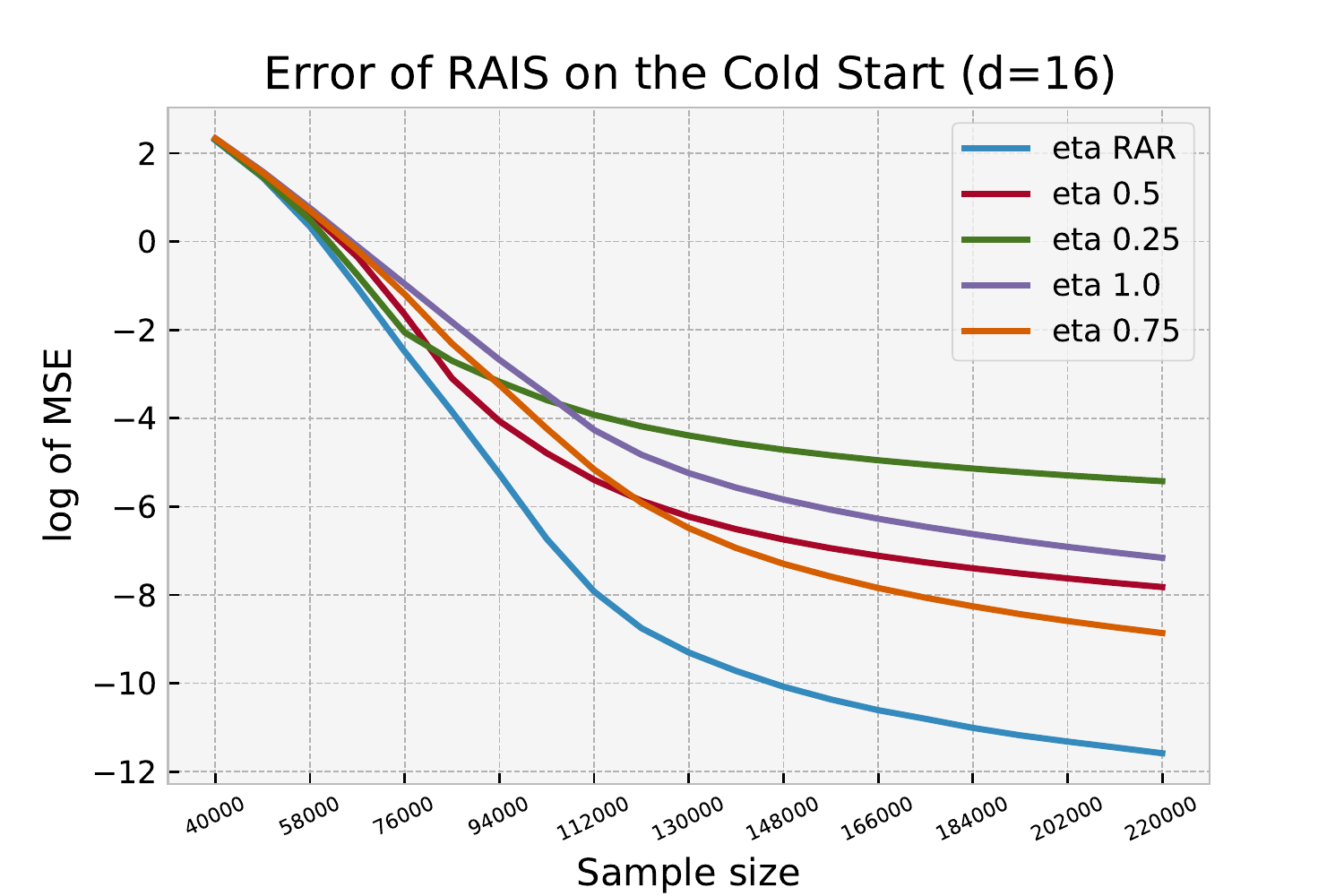}\includegraphics[scale=0.27]{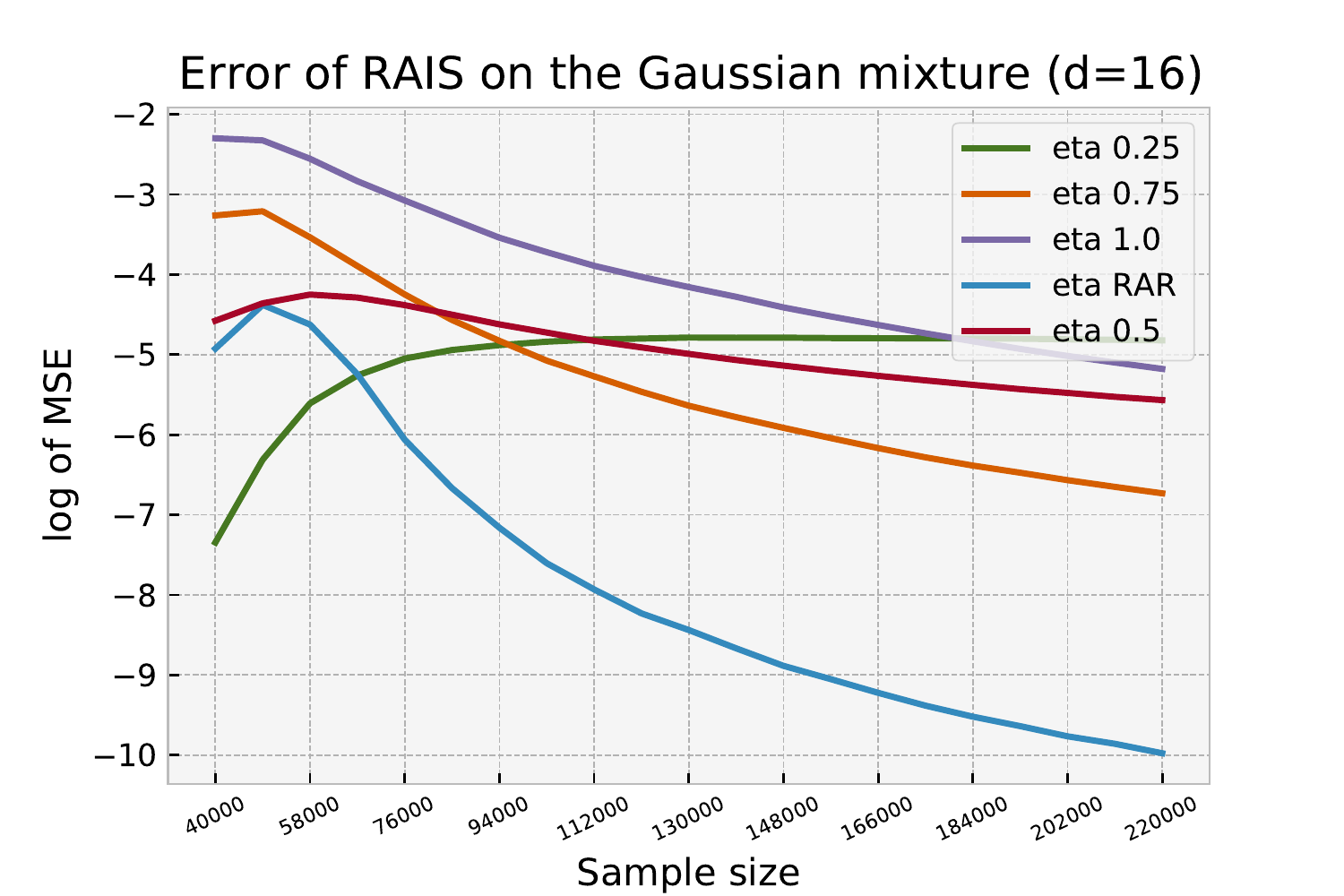}\includegraphics[scale=0.27]{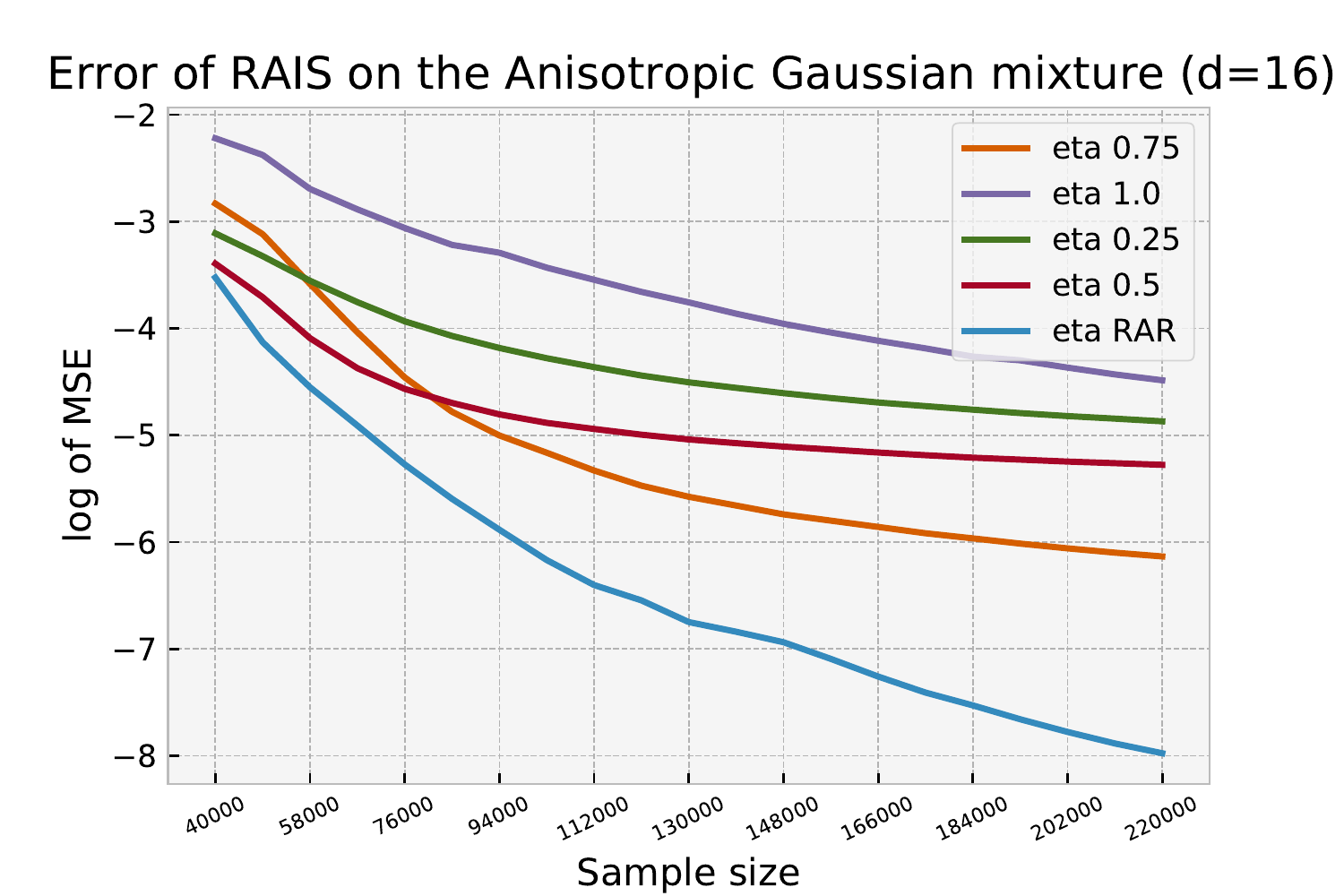} \\
  \caption{Logarithm of the average squared error for \Cref{algo:RIS} for constant values of $\eta$ or \Cref{algo:ADA}, computed over 50 replicates. 
  }\label{fig:toy_error}
\end{figure}

Figure \ref{fig:toy_error} shows that the sequence $(\eta_{k})_{k\geq 1}$ has a strong influence on the outcome of the procedure and an adaptive choice of $(\eta_k)_{k\geq 1}$ (reflecting the quality of the current proposal) leads to a substantial improvement. An explanation of the performance of \Cref{algo:RIS} along with \Cref{algo:ADA} can be found in Figure \ref{fig:toy_bp} where we see that in all cases, at the beginning of the algorithm when the policy is poor, the value of $\eta_k$  is automatically set to a small value (leading to a uniformization of the weights); and when the policy becomes better the value of $\eta_k$ converges to $1$.


\begin{figure}[h]
	\centering
		\includegraphics[scale=0.27]{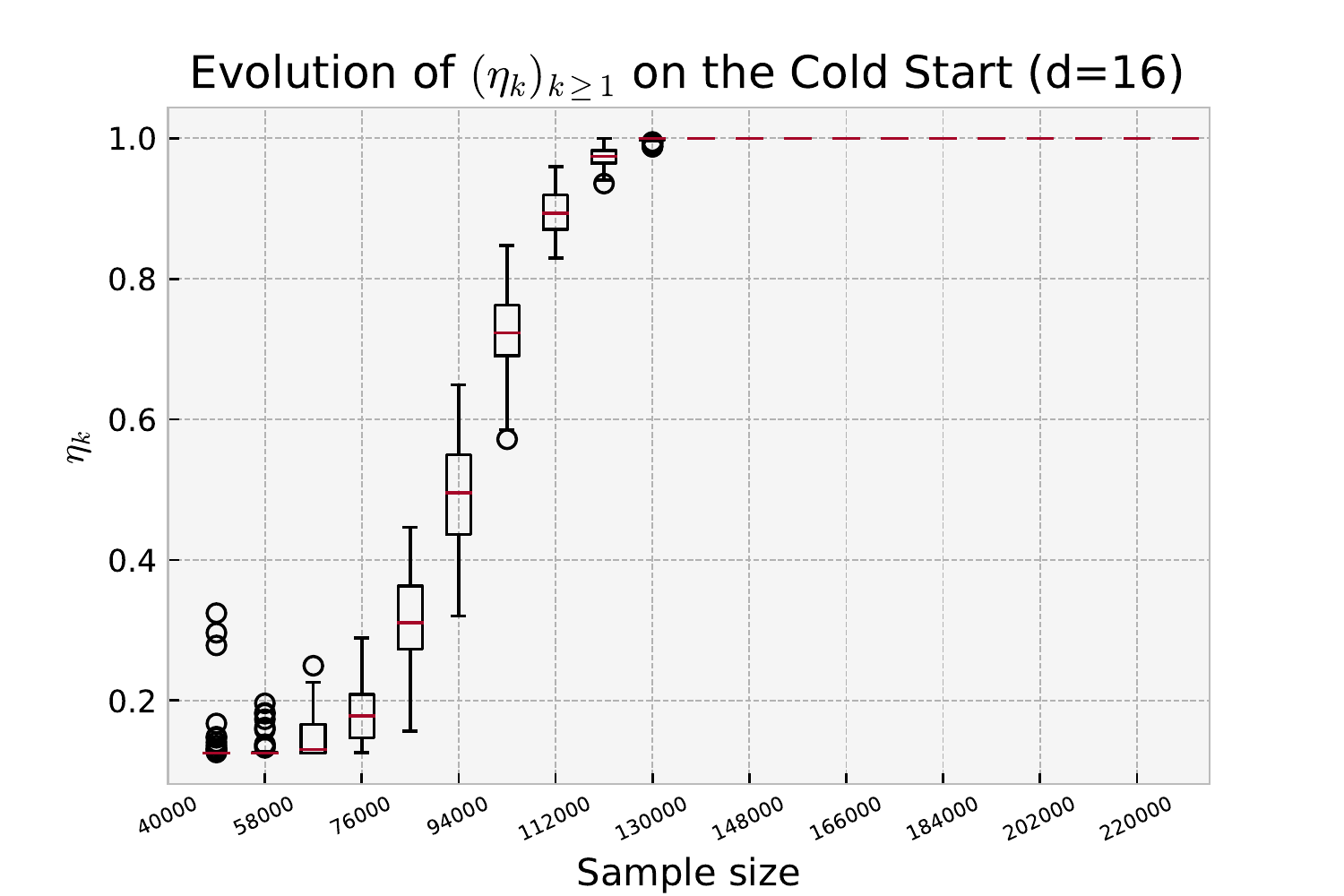}\includegraphics[scale=0.27]{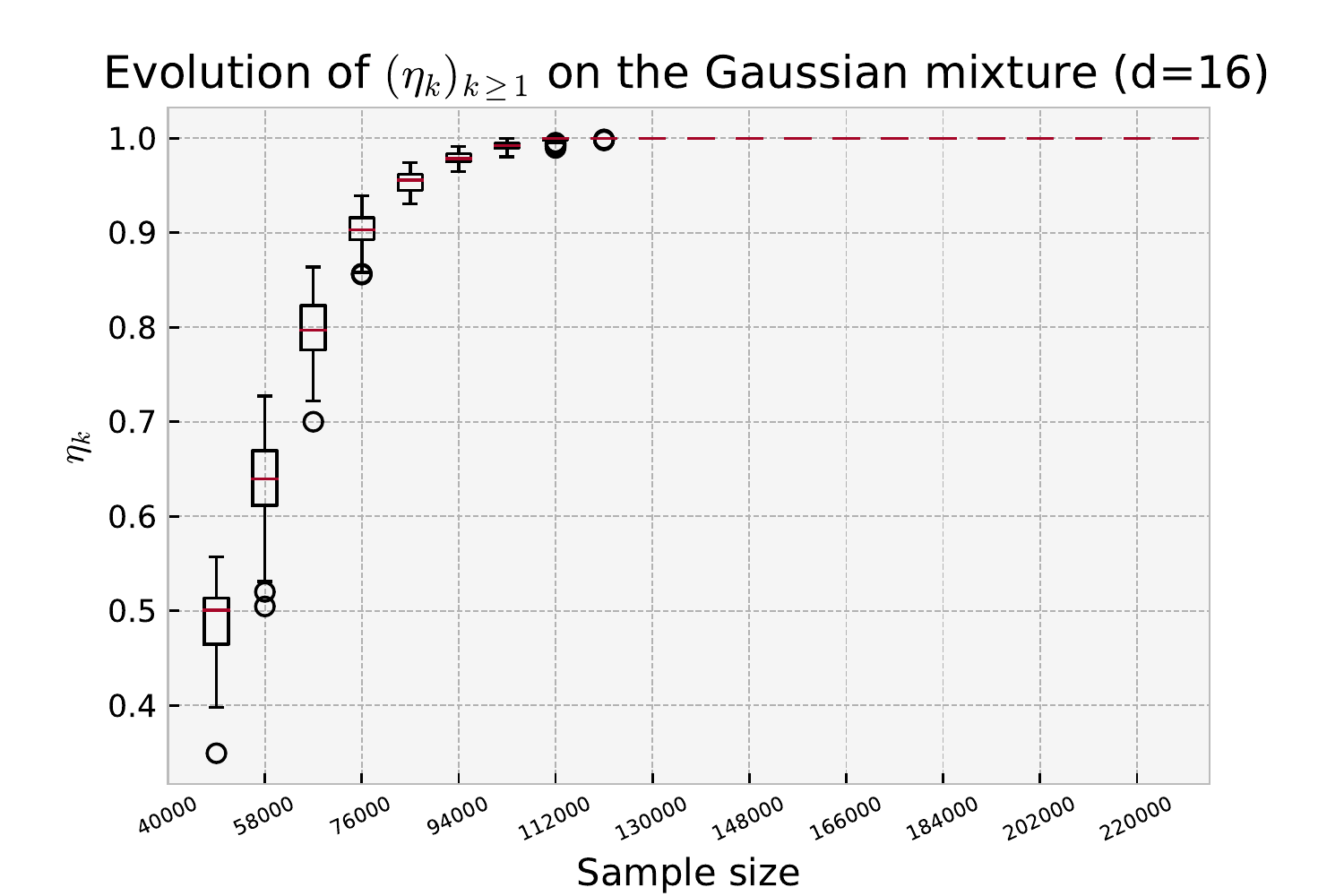}\includegraphics[scale=0.27]{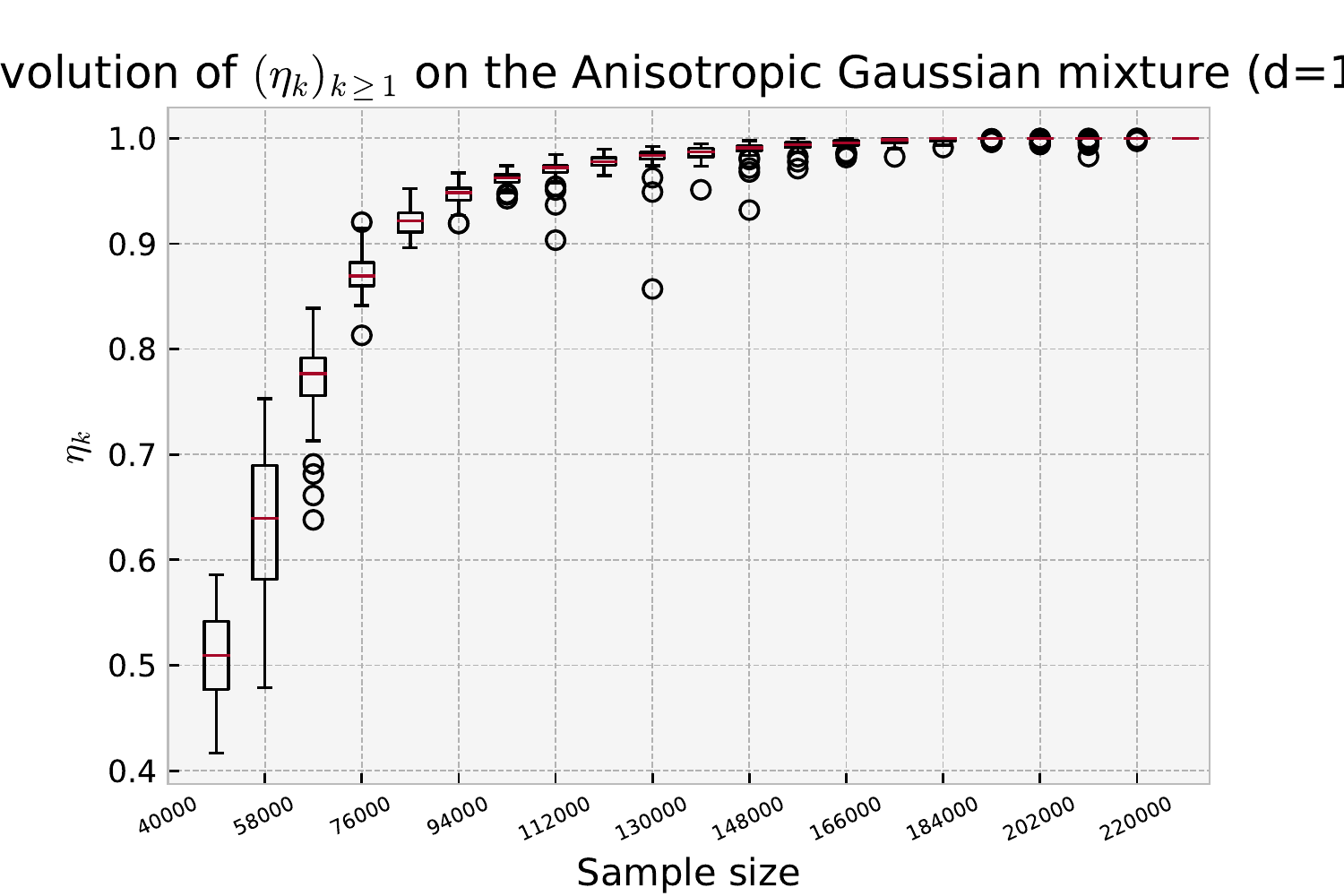} 

  \caption{Boxplot of the values of $(\eta_{k,\alpha})_{k\ge1}$ obtained from \Cref{algo:ADA}, with $\alpha=0.5$.
  }\label{fig:toy_bp}
\end{figure}

\subsection{Bayesian Logistic Regression}

We consider the Bayesian Logistic Regression setting of \cite{DBLP:journals/corr/abs-1206-4665}, also considered in the recent bayesian inference literature 
\cite{liu2016stein,daudel2020infinite}. 
More precisely, we observe a dataset $\data = \lrcb{c_i, \boldsymbol{z}_i}_{i\in I}$ of points $\boldsymbol{z}_i \in \Rset^L$ and binary class labels $c_i \in \lrcb{-1, 1}$, for $i\in I$. 
We assume the following model
: $p_0(\beta) = \mathrm{Gamma}(\beta; a, b)$, and for any $1\le l\le L$ and $1\le i\le I$,
\begin{align*}
  p_0(\omega_\ell| \beta) = \mathcal{N}(\omega_\ell; 0, \beta^{-1}, \;
  p(c_i = 1 | \boldsymbol{z}_i, \boldsymbol{\omega}) = \frac{1}{1 + e^{- \boldsymbol{\omega}^T \boldsymbol{z}_i}},
\end{align*}
where $a$ and $b$ are hyperparameters (shape and inverse scale) that are 
fixed to $a = 1$ and $b = 0.01$, and $\Gamma,\mathcal{N}$ denote the Gamma and Gaussian distribution respectively. The parameter vector is then $[\boldsymbol{\omega},\beta]\in \Rset^{d}$, with $\beta\in \Rset^{+}$ and $d=L+1$. Given a new data point $\boldsymbol{z}_{\mathrm{new}}$, we are interested in predicting the label $c_{\mathrm{new}}$ using the \textit{posterior predictive distribution}
$
p(c_{\mathrm{new}}|\boldsymbol{z}_{\mathrm{new}},\data) = \int p(c_{\mathrm{new}}|\boldsymbol{z}_{\mathrm{new}}, x) p(x|\data) \rmd x, 
$
which plays the role of the target distribution. 

Since the posterior predictive distribution is intractable for this choice of model, we resort to \Cref{algo:RIS} to approximate this quantity. 
We consider two main regimes for the regularization scheme in our numerical experiments:
  (i) constant 
  policy $\eta_k = 1$ for all $k \geq 1$: this case corresponds to Safe Adaptive Importance Sampling algorithm \cite{delyon2021safe}, with no bias but the highest variance,
  (ii) adaptive  
  policy $(\eta_{k,\alpha})_{k\geq1}$ from \Cref{algo:ADA}, with $\alpha \in \lrcb{0.3, 0.2, 0.1, 0.08}$.
Here, the bandwith of the kernel estimate $h_k$ is chosen to be proportional to $k^{-1/(4 + d)}$. This corresponds to the optimal choice in nonparametric estimation when the target density $f$ is at least 2-times continuously differentiable and the kernel has order 2 \cite{10.1214/aos/1176345969}.\ak{ca veut dire quoi has order 2?} We also choose $\lambda_k$ proportional to $1/\sqrt{k}$ for all $k \geq 1$. Furthermore, $q_0$ is chosen as a Gaussian distribution with mean $0. \boldsymbol{u}_d$ 
and variance $5. \boldsymbol{I}_d$, where $\boldsymbol{u}_d$ is the $d$-dimensional vector whose coordinates are all equal to $1$ and $\boldsymbol{I}_d$ is the identity matrix. Finally, our algorithms are initialised by sampling independently $2000$ particles from $q_0$. We further sample independently $200$ particles from $q_k$ at each time $k \geq 1$, for a total computational budget of 61800 particles (i.e $n = 300$ iterations).

We test these algorithms on the Waveform dataset (link provided in the Appendix)
composed of 5000 datapoints  in $\Rset^{L}$ with $L=21$, i.e. $d=22$
. 
Figure \ref{fig:acc:blr} displays the average accuracy of the predictions on the testing set and the  averaged values of the sequence $(\eta_{k,\alpha})_{k\geq1}$ for $\alpha \in \lrcb{0.3, 0.2, 0.1, 0.08}$, where the average is made over 100 replicates of the experiment.

\begin{figure}[h]
  \centering
   \includegraphics[scale=0.29]{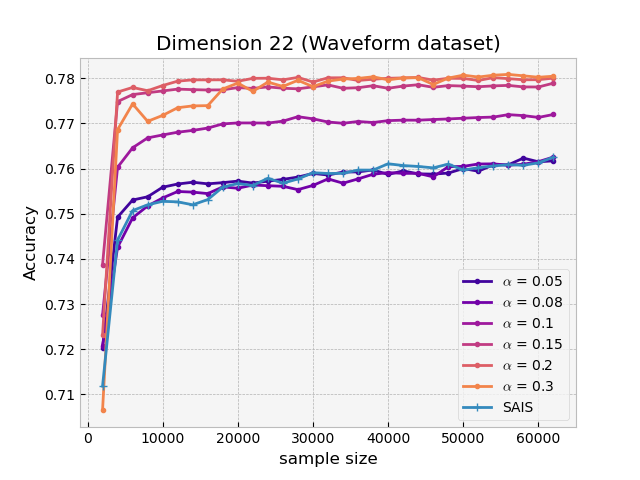}  \includegraphics[scale=0.29]{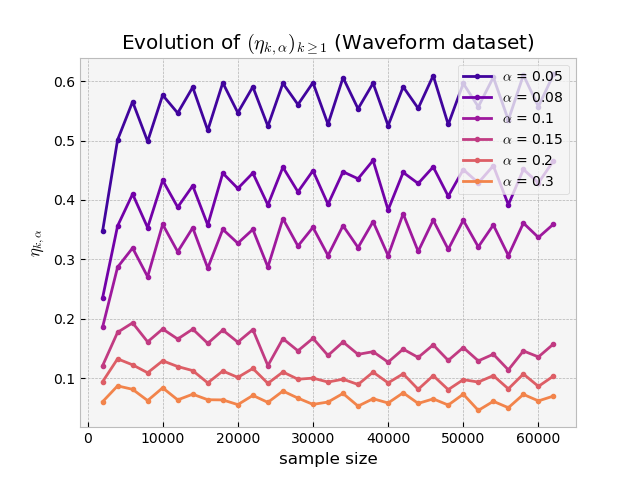} 
  \caption{Left plot: Average accuracy over 100 trials of different learning policies $(\eta_{k,\alpha})_{k \geq 1}$ for
Bayesian Logistic Regression on the Waveform dataset. Right plot: Averaged values of the learning policy $(\eta_{k,\alpha})_{k \geq 1}$ associated to each choice of $\alpha$.}\label{fig:acc:blr}
\end{figure}

As stated in \Cref{prop:tuning_eta}(ii) and can be seen on the right plot of \Cref{fig:acc:blr}, the lower the value of $\alpha$, the higher the value of $\eta_{k,\alpha}$ at time $k$. Recalling that the level of confidence in high values of the learning policy is expressed through the parameter $\alpha$, observe then on the Left plot of \Cref{fig:acc:blr} that a proper tuning of the parameter $\alpha$ allows us to outperform the Safe Adaptive Importance Sampling case (\Cref{algo:RIS} with $(\eta_k)_{k\ge 1}$ constant and equal to $1$), the case $\alpha = 0.2$ yielding the best results here overall in terms of speed and accuracy.

%


%% file: conclu.tex
\vspace*{-0.2cm}
\section{Conclusion}

We proposed a new algorithm for AIS, that regularizes the importance weights by raising them to a certain power $\eta$ that plays the role of a regularizer. Relating this algorithm to mirror descent on the space of probability distributions, we prove that it enjoys a uniform convergence guarantee, under mild assumptions on the regularity of the target distribution and safe density, and hyperparameters of the algorithm. Furthermore, by proposing an adaptive way to schedule the regularization, we provide a method that achieves the best numerical performances, compared to a constant regularization or classical safe adaptive importance sampling. The results that we obtained could be very impactful for machine learning tasks where importance sampling is used, e.g. bayesian inference, but also stochastic optimization or reinforcement learning. Future work includesa non asymptotic analysis of our scheme.

%% file: appendix.tex
\section{Proofs of the lemmas}

\subsection{Proof of \Cref{lem:bias_variance}}
\label{sec:proof_lemma_bias_variance}

Because $q$ dominates $f$, we have $\mathbb{E} [W(X)] = 1$. The first inequality is due to Jensen's inequality: $1 = \mathbb E [ W(X)] ^{\eta} \geq \mathbb E [ W(X)^{\eta}] $.  When $W(X)$ is not a constant, equality holds if and only if $\eta = 1$.

For the second inequality, write
\begin{equation*}
	\Var[W(X)^\eta] \leq\Var[W(X)^\eta] + (\mathbb E [  W(X)^\eta ]  - 1) ^2 = \mathbb E [ (W(X)^\eta - 1 )^2] \leq \mathbb E [ (W(X)  - 1 )^2]   .
\end{equation*}
The first inequality has already been justified. The second inequality holds because $|w^\eta - 1|\leq |w-1|  $ for all $w\geq 0$.
\qed


\subsection{Proof of  \Cref{lemma:tv}}\label{sec:proof_lemma_tv}

Recall that the general definition of the Kullback-Leibler divergence is given by
$$
\kl[f][q] = \int f \log \lr{\frac{f}{q}} + \int q - \int f \eqsp.
$$
Note that it extends the definition given in \Cref{lemma:tv} to unnormalized densities.
Let $q$ be a probability density function and set $\tq=f^{\eta} q ^{1-\eta}$. Then, using that $\log(u) \leq u - 1$ for all $u> 0$, we have that
\begin{align}
	\label{eq:normalised}
	KL \lr{f \bigg\| \frac{\tq}{\int \tq} } &= \int f \log \lr{\frac{f}{\tq} \cdot \int \tq} \nonumber \\
	& = \int f \log \lr{\frac{f}{\tq}} + \log\lr{\int \tq} \nonumber \\
	&\leq \kl[f][\tq] \eqsp.
\end{align}
Furthermore, by definition of $\tq$, it holds that
\begin{align*}
	\kl[f][\tq] & = \int \log \left( \frac{f}{f^{\eta} q ^{1-\eta}}  \right) f + \int  f^{\eta} q^{1-\eta}  - 1\\
	& =  (1-\eta) \int \log(f /  q    ) f + \int  f^{\eta} q ^{1-\eta}  - 1\\
	& =  (1-\eta) KL (f \|  q   )  + \int  f^{\eta} q ^{1-\eta}  - 1 \\
	& \leq (1-\eta) \kl[f][q] \eqsp,
\end{align*}
where the last inequality results from Jensen's inequality applied to the convex function $u \mapsto u^\eta$:
\begin{align*}
	\int f^{\eta} q^{1-\eta} =  \int  q \lr{ \frac{f}{q}}^{\eta} \leq  \lr{\int  f}^\eta  = 1 \eqsp.
\end{align*}
Combining with \eqref{eq:normalised} and letting $(q^*_k)_{k\geq 1}$ be defined by \eqref{eq:emd} starting from an initial probability density function $q_1$, by recursion we have for all $n \in \nstar$,
$$
\kl[f][q^*_{n+1}] \leq \kl[f][q_1] \prod_{k=1}^{n}  ( 1-\eta_k ).
$$
By applying Pinsker's inequality, we finally obtain
\begin{align*}
	\int | f - q^*_{n+1} | \leq \sqrt{ 2 \kl[f][q_1] }  \prod_{k=1}^{n}  ( 1-\eta_k )^{1/2 }.
\end{align*}
\qed

\paragraph{Convergence rates obtained from \Cref{lemma:tv}.}
Assuming that $ \kl[f][q_1] <+\infty$ and noticing that
$$
\log \lr{\prod_{k=1}^{n}  ( 1-\eta_k )^{1/2 }} \leq - \frac{1}{2} \sum_{k=1}^{n} \eta_k,
$$
we get the following convergence rates:
\begin{itemize}
	\item taking $\eta_k = c  / k$ with $0< c< 1$ yields
	$  \int | f - q^*_{n+1} | =   O(   n^{-c / 2}   ),$
	\item taking $\eta_k = c / {k}^\beta$ with $0< c< 1$ and $\beta\in [0,1)$ yields
	$\int | f - q^*_{n+1} | = O(  \exp( - C   n^{(1-\beta) }   )     ,$ with  $ C = c  / (2(1-\beta))$.
\end{itemize}

\section{Deriving \eqref{eq:emd} from an optimisation
	perspective}
\label{sec:NoteEMD}

One way to approximate an unknown probability density is to formulate an optimisation problem over a certain space of distributions, as it is typically done in \textit{variational inference}. The common choice in variational inference then often corresponds selecting the Kullback-Leibler divergence and to try to find
\begin{align*}
	q^\star = \arginf_{q \in \mathcal{Q}}~ \kl \eqsp,
\end{align*}
where $\mathcal{Q}$ is a valid set of probability densities on $(\rset^d, \mathcal{B}(\rset^d))$, $\mathcal{B}(\rset^d)$ denoting the Borel $\sigma$-field of $\rset^d$, and where $\kl[f][q]$ stands for the Kullback-Leibler divergence between $f$ and $q$, i.e $\kl[f][q] = \int \log\lr{f/ q} f$.

Following the approach of \cite{dai2016provable}, one way to solve this optimisation problem is to resort to the \textit{entropic mirror descent} algorithm applied to the objective function $q \mapsto \kl$. When $\mathcal{Q}$ corresponds to the set of probability density functions on $(\rset^d, \mathcal{B}(\rset^d))$, this algorithm admits a closed-form solution and generates a sequence $(q_{k})_{k\ge 1}$ in $\Q$ satisfying \eqref{eq:emd}.

To see this, let us start with a preliminary result. Let $h$ be a real-valued measurable function defined on $\rset^d$ such that $\int \exp(-h) \rmd \lambda <\infty$. For any probability density function $f$ on $\rset^d$ such that $ \int |h + \log(f)| f\rmd \lambda  < \infty$, define
\begin{align*}
	\Psi(f) = \int (h  +  \log ( f ) )  f\rmd \lambda .
\end{align*}

\begin{lemma}
The minimum of the function $\Psi$ is attained for $f \propto \exp(-h) $.
\end{lemma}

\begin{proof}
	By applying Jensen's inequality to the convex function $u \mapsto \exp(- u)$, we obtain
	\begin{align*}
		\exp(- \Psi(f)) \leq \int \exp \lr{- \lrb{ h + \log ( f ) }}  f \rmd \lambda \eqsp.
	\end{align*}
	Thus, we have that
	\begin{align*}
		\Psi(f) \geq - \log \lr{ \int \exp \lr{- h} \rmd \lambda} \eqsp,
	\end{align*}
	where the r.h.s does not depend on $f$ and equality is attained whenever $f \propto \exp(-h) $ almost everywhere. 
\end{proof}


Next, we rewrite \eqref{eq:emd} as an Entropic Mirror Descent step.
For any $x \in \Rset^d$ probability density $q \in \mathcal{Q}$, we set $h_q(x) = \log(q(x)/f(x)) + 1$. Given a probability density $q^*_k$ and $\eta_k > 0$, one iteration of the (Infinite-Dimensional) Entropic Mirror Descent algorithm applied to the objective function $q \mapsto \kl$ with a learning rate $\eta_k$ corresponds to finding
\begin{align*}
	q^*_{k+1} = \argmin_{q \in \mathcal{Q}} ~ \eta_k \int h_{q^*_k}(x) q(x) \rmd x + \kl[q][q^*_k] \eqsp.
\end{align*}
In this expression, which is called the proximal form of the Entropic Mirror Descent, the function $h_{q^*_k}$ plays the role of the gradient of $\kl$ w.r.t the probability density $q^*_k$ (here, it corresponds to its Fréchet differential). Based on the previous paragraph, we deduce that if $\mathcal{Q}$ corresponds to the set of probability density functions on $(\rset^d, \mathcal{B}(\rset^d))$, then
\begin{align*}
	q^*_{k+1} = \frac{ q^*_k(x) e^{- \eta_k h_{q^*_k}(x)}}{\int  q^*_k(x') e^{- \eta_k h_{q^*_k}(x')} \rmd x'} \propto f^{\eta_k}(x)  q^*_k(x) ^ {1-\eta_k} \eqsp,
\end{align*}
that is, we recover \eqref{eq:emd}.

Under minimal assumptions, the convergence towards $f$ can be established with a known convergence rate for an appropriate choice of learning policy $(\eta_k)_{k \geq 1}$.
\begin{lemma}\label{lemma:cv:emd}
	Let $(\eta_k)_{k \geq 1}$ be a sequence of positive learning rates and let $\mathcal{Q}$ be set of probability density functions on $(\rset^d, \mathcal{B}(\rset^d))$. Let $q_1 \in \mathcal{Q}$ and let the sequence $(q^*_k)_{k\geq 1}$ be defined by \eqref{eq:emd}. Assume that $x \mapsto h_q(x)$ is bounded by a positive constant $L$ for all $x \in \Rset^d$ and $q \in \mathcal{Q}$. Then, for all $n \in \nstar$, we have
	$$
	\KL \lr{\sum_{k=1}^n \frac{\eta_k q^*_k}{\sum_{k'=1}^n \eta_{k'}}  \bigg\| f } \leq \frac{\sum_{k=1}^n \eta_k^2 L^2 /2}{\sum_{k=1}^n \eta_{k}} + \frac{\kl[f][q_1]}{\sum_{k=1}^n \eta_{k}} \eqsp .
	$$
	In particular, taking $\eta_k = c_0/ \sqrt{k}$ with $c_0 > 0$ yields an $O(\log(n)/\sqrt{n})$ convergence rate. If the total number of iterations $n$ is known in advance, setting $\eta_k = c_0/\sqrt{n}$ for all $k=1 \ldots n$  with $c_0 > 0$ yields an $O(1/\sqrt{n})$ convergence rate.
\end{lemma}
The proof of this result can be adapted from \cite[Theorem 4.2]{MAL-050}. It is provided here for the sake of completeness.

\begin{proof}
	For all $k \geq 1$, set $\Delta_k = \kl[q^*_k]$. By convexity of the function $u \mapsto u \log u$, we have
	\begin{align*}
		\Delta_k  & = \int \log \lr{\frac{q^*_k(x)}{f(x)}} q^*_k(x) \rmd x \\
		& \leq \int \lrb{ \log \lr{\frac{q^*_k(x)}{f(x)}} + 1} (q^*_k(x)- f(x)) \rmd x \eqsp,\\
		& = \int h_{q^*_k}(x) (q^*_k(x) - f(x)) \rmd x \eqsp.
	\end{align*}
	Since the integral of any constant w.r.t $q^*_k - f$ is null, we deduce
	\begin{align*}
		\eta_k \Delta_k  & \leq \int \log \lr{\frac{q^*_k(x)}{q^*_{k+1}(x)}} (q^*_k(x)- f(x)) \rmd x \\
		&=\int \log\lr{\frac{q^*_k(x)}{q^*_{k+1}(x)}} q^*_k(x) \rmd x - \int \log\lr{\frac{q^*_k(x)}{q^*_{k+1}(x)}} f(x) \rmd x\\
		&= \int \log\lr{\frac{q^*_k(x)}{q^*_{k+1}(x)}} (q^*_k(x) - q^*_{k+1}(x)) \rmd x -\kl [q^*_{k+1}][q^*_k] \\
		&\qquad + \lrb{ \kl[f][q^*_k]-\kl[f][q^*_{k+1}]}
	\end{align*}
	Let us consider the first term of the r.h.s. of the latter inequality. We have that
	\begin{align*}
		\int \log\lr{\frac{q^*_k(x)}{q^*_{k+1}(x)}} (q^*_k(x) - q^*_{k+1}(x)) \rmd x &= \eta_k \int h_{q^*_k}(x) (q^*_k(x) - q^*_{k+1}(x)) \rmd x \\
		& \leq \eta_k L \int |q^*_k - q^*_{k+1}| \eqsp.
	\end{align*}
	since by assumption $h_{q^*_k}$ is bounded by $L$. Additionally, we have by Pinsker's inequality that
	$$
	-\kl [q^*_{k+1}][q^*_k] \leq -\frac12 \lr{\int |q^*_k - q^*_{k+1}|}^2  \eqsp.
	$$
	Now combining with the fact that $\eta_k L a -a^2/2 \leq (\eta_k L)^2/2$ for all $a\geq0$, we get:
	$$
	\int \log\lr{\frac{q^*_k(x)}{q^*_{k+1}(x)}} (q^*_k(x) - q^*_{k+1}(x)) \rmd x -\kl [q^*_{k+1}][q^*_k] \leq \frac{(\eta_k L)^2}{2}
	$$
	and as a consequence we deduce
	\begin{align*}
		\eta_k \Delta_k  & \leq \frac{(\eta_k L)^2}{2} + \lrb{ \kl[f][q^*_k]-\kl[f][q^*_{k+1}]} \eqsp.
	\end{align*}
	Finally, as we recognize a telescoping sum in the right-hand side, we have
	$$
	\sum_{k=1}^n \eta_k \Delta_k \leq \sum_{k=1}^n \eta_k^2 L^2 /2 + \kl[f][q_1] \eqsp
	$$
	that is, by convexity of the mapping $q \mapsto \kl$,
	$$
	\KL \lr{\sum_{k=1}^n \frac{\eta_k q^*_k}{\sum_{k'=1}^n \eta_{k'}}  \bigg\| f } - \kl[f] \leq \frac{\sum_{k=1}^n \eta_k^2 L^2 /2}{\sum_{k=1}^n \eta_{k}} + \frac{\kl[f][q_1]}{\sum_{k=1}^n \eta_{k}} \eqsp
	$$
	Then, notice that taking $\eta_k = \eta_0/ \sqrt{k}$ with $\eta_0 > 0$ yields an $O(\log(n)/\sqrt{n})$ convergence rate and that setting $\eta_k = \eta_0/\sqrt{n}$ for all $k=1 \ldots n$ yields an $O(1/\sqrt{n})$ convergence rate.
\end{proof}

\section{Proof of Proposition \ref{prop:initial_bound}}\label{sec:proof_initial_bound}

The proof is organized in three parts. First we provide high-level results related to Freedman's inequality for martingales.  Then we provide some intermediary technical results, and finally we conclude with the proof of \Cref{prop:initial_bound}.

\subsection{Bernstein inequalities for martingale processes}\label{sec:append_freedman}

The two following propositions can be found in \cite{delyon2021safe}.

\begin{prop}\label{th:freedman}
	Let $(\Omega, \mathcal F , (\mathcal F _{k})_{k\geq 1}, \mathbb P)$ be a filtered space. Let $(Y_k)_{1 \le k\le n} $ be real valued random variables such that
	\begin{align*}
		&\mathbb E [Y_k| \mathcal F _{k-1}] = 0,\quad\text{for all} ~k
		=1,\ldots, n
		\eqsp.
	\end{align*}
	Then, for all $t\geq 0$ and all ${v},{m}>0$,
	\begin{align*}
		\mathbb P\Bigg(\Big|\sum_{k=1}^nY_k\Big|\geq t, \,\max_{k=1,\ldots, n}|Y_k |\leq {m},\,
		\sum_{k=1}^n &\mathbb E[Y_k^2|\mathcal F _{k-1}]\leq {v}\Bigg) \leq 2\exp\left(-\frac{t^2}{2({v}+t{m}/3)} \right) \eqsp.
	\end{align*}
\end{prop}

\begin{prop}\label{cor:freedmanx}
	Let $(\Omega, \mathcal F , (\mathcal F _{k})_{k\geq 1}, \mathbb P)$ be a filtered space. Let $(Y_k)_{k\geq 1} $ be a sequence of real valued stochastic processes defined on
	$\mathbb R^d$, adapted to $(\mathcal F_k)_{k\geq 1}$, such that
	for any $x\in\mathbb R^d$,
	\begin{align*}
		&\mathbb E [Y_k(x)| \mathcal F _{k-1}] = 0,\quad \text{for all } k\geq 1 \eqsp.
	\end{align*}
	Consider $\epsilon>0$ and let  $(\tilde Y_k)_{k \geq 1} $ be another $(\mathcal F_k)_{k\geq 1}$-adapted sequence
	of nonnegative stochastic processes defined on $\mathbb R^d$ such that for all $k\geq 1$ and $x\in\mathbb R^d$
	\begin{align*}
		&\sup_{\|y\|\le\epsilon}|Y_k(x+y)-Y_k(x)|\le \tilde Y_k(x) \eqsp.
	\end{align*}
	Let $n\geq 1$ and assume that for some $A\ge 0$ and some set $\Omega_1\subset \Omega$, there exist $m, v, \tau \in \Rset^{+}$ such that  for all $\omega\in\Omega_1$
	and $\|x\|\le A$,
	\begin{align}
		& \max_{k=1,\ldots, n} |Y_k(x)|\le m\label{thmm}\\
		&\sum_{k=1} ^ n \mathbb E\big[Y_k(x)^2|\mathcal F_{k-1}]\le v\\
		&\sum_{k=1} ^ n \mathbb E\big[\tilde Y_k(x)|\mathcal F_{k-1}\big]\le \tau \eqsp.
	\end{align}
	Then, for all $t\geq 0$,
	\begin{align*}
		&\mathbb P\Big(\sup_{\|x\|\le A}\big|\sum_{k=1} ^ n Y_k(x)\big|>t+\tau,\Omega_1\Big)\le 4(1+2A/\varepsilon)^d \exp\left(-\frac{t^2}{8(\tilde v+2mt/3)} \right) \eqsp,
	\end{align*}
	with $\tilde v=\max(v,2m\tau).$
\end{prop}

Recall also that the \textit{predictable quadratic variation} \cite{bercu2015concentration} of a martingale $\sum_{k=1} ^n \beta_k$ is given by
\begin{align*}
	\sum_{k=1} ^n \mathbb E [ \beta_k ^2| \mathcal F_{k-1} ].
\end{align*}
This quantity is important as it appears as an essential factor in the Bernstein inequalities above. In particular, the predictable quadratic variation of the function $M_n$ defined in \eqref{eq:mn} can be found in \Cref{prop:mg} below.

\begin{prop}\label{prop:mg}
	Suppose that the kernel function $K:\Rset^{d}\to \Rset^+$ is bounded.  For $n\geq 1 $, let $\mathcal F_{n} =\sigma(X_1,\ldots, X_{n})$ be the $\sigma$-algebra generated by the random variables $X_1,\dots,X_n$, and $\mathcal F_0 = \emptyset$. Then, for each $x\in \mathbb R^d$, $\left( nM_n (x) \right)_{n\geq 1} $ is a $(\mathcal F_n)_{n\geq 1}$-martingale  with predictable quadratic variation $\sum_{k=1}^n V_{\eta_k,h_n} ( x )$ where $V _{\eta,h} (x) = \int   K_{h} (x  - y) ^ 2 f ^{2\eta}(y)  q^{1-2\eta}(y) dy   - f\star  K_{h} (x) ^ 2$.
\end{prop}
\begin{proof}
	To lighten notation, we set $g(y) = K_{h_n} (x-y)$ for the next few lines.
	Because $q_k$ is positive on $\mathbb R^d$ and $\mathcal F_k$-measurable, $\mathbb E [ W_k^{\eta_k}g(X_k) \mid\mathcal F_{k-1}] = \int  f^{\eta_k} q_{k-1}^{1-\eta_k}  g $. The formula for the predictable quadratic variation follows in the same way.
\end{proof}

\subsection{Technical results}\label{sectechnical_results}\ak{Z1 et Z2 ca s'appelle des martingales increments cest ca? peut être qu'on peut donner ce titre à cette section? genre control of the martingale increments}

We start with some notation.

\textbf{Notation.} Recall from \Cref{sec:results} that for all $n \geq 1$ and all $x\in \mathbb R^d$,
\begin{align*}
    &f_{n}(x)= \frac{N_n(x)}{D_n} \\
	&N_{n} (x) = n^{-1} \sum_{k=1} ^{n} W_{k} ^{\eta_k} K_{h_n} (x  - X_{k}) \\
	&D_{n} = n^{-1} \sum_{k=1} ^{n} W_{k} ^{\eta_k} \eqsp.
\end{align*}

In addition, for all $k \geq 1$ and all $x \in \Rset^{d}$, we introduce the helpful notation
\begin{align*}
	&\tilde{f}_k(x) = f^{\eta_k}(x)q_{k-1}^{1-\eta_k}(x),\\
	&Z^{(1)}_k(x) = W^{\eta_k}_k K_{h_k}(x - X_k) - (\tilde{f}_k \star K_{h_k})(x),\\
	&Z^{(2)}_k = W^{\eta_k}_k - \int \tilde{f}_k.
\end{align*}

The three following propositions focus on convergence results for $Z^{(1)}_k$ and $Z^{(2)}_k$. \ak{on peut raconter une meilleure histoire?}

\begin{prop}\label{prop:unif_bound_weights}
	Under \ref{cond:sequences} and \ref{cond:lowerbound} we have that almost-surely,
	\begin{align*}
		\forall k = 1,\ldots, n,\qquad  0 \leq W_k^{\eta_k}  \leq  (  c \lambda_{n}  )^{-1} \eqsp.
	\end{align*}
\end{prop}
\begin{proof}
	Note that by definition of $q_k$ in \eqref{eq:sampler_update} combined with \ref{cond:lowerbound}, we have that for all $k = 1, \ldots , n$ and all $x \in \rset^d$,
	\begin{equation}\label{eq:lb_qk}
	q_k(x) \geq \lambda_k q_0(x) \geq \lambda_k c f(x) \eqsp.
\end{equation}
	As a consequence, $ 0 \leq W_k  \leq 1/( c\lambda _ {k-1} )\leq 1/( c\lambda _ {n} ) $ using that $(\lambda_k)_{k\geq 1}$ is nonincreasing under \ref{cond:sequences}. Note also that \eqref{eq:lb_qk}, we have that $\lambda _ n \in (0,1/c]$. Thus, we can write, for all $k = 1, \ldots , n$,
	$(c\lambda_n)^{-\eta_k} \leq (c\lambda_n)^{-1} $, which implies the stated result.
\end{proof}


\begin{prop}\label{prop:intial_bound_z2}
	Under \ref{cond:sequences} and \ref{cond:lowerbound} we have that almost-surely
	\begin{align*}
		n^{-1}  \sum_{k=1} ^n   Z_{k} ^{(2)} = O \left( \sqrt { \frac{\log(n)  } { n \lambda_n} } \right).
	\end{align*}
\end{prop}


\begin{proof}
	We first show that without loss of generality, we can derive the proof assuming that the sequence  $ (\eta_k)_{k\geq 1} $ is valued in $ (1/2,1]$.
	From \Cref{prop:unif_bound_weights}, it easy to see that for all $k=1,\ldots, n$,
	\begin{align}\label{eq:bound_unif_z2}
		|Z_{k} ^{(2)}| = |W_k^{\eta_k} - \mathbb E [W_k^{\eta_k}| \mathcal F_{k-1} ] | \leq  (  c \lambda_{n}  )^{-1} \eqsp.
	\end{align}
	Since the sequence $(\eta_k)_{ k\geq 1} $ goes to $1$ under \Cref{cond:sequences}(iii), there is an integer $k_0\geq 1$ such that $\eta_k \in (1/2, 1]$ for all $k > k_0$.  Hence, whenever $n>k_0$, we have
	$$ \sum_{k=1} ^n  Z_{k} ^{(2)} =\sum_{k=1} ^{k_0}    Z_{k} ^{(2)}  + \sum_{k=k_0+1} ^n   Z_{k} ^{(2)}.  $$
	By \eqref{eq:bound_unif_z2}, the first term  is bounded by $k_0 / (c\lambda_{k_0} ) $ and hence has a negligible contribution in the bound we need to establish. The only term that matters is then the second one. Hence, from now on we assume that $ (\eta_k)_{k\geq 1} $ is valued in $ (1/2,1]$. The goal will be to apply \Cref{th:freedman} to $Y_k(x) = Z_{k}^{(2)}(x)$.
	Using that $0< 2\eta_k - 1 \leq 1$ and \Cref{prop:unif_bound_weights}, we can write
	\begin{align*}
		\mathbb E[Z_k ^{(2)2} |\mathcal F _{k-1}] &\leq \mathbb E[ f(X_k ) / q_{k-1} (X_k) )^{2\eta_k} |\mathcal F _{k-1}]\\
		& =  \int   f(x )^{2\eta_k }  q_{k-1} (x)^{1 - 2\eta_k } \, \rmd x\\
		&= \int  ( f(x )/ q_{k-1} (x) )  ^{2\eta_k -1 }  f (x)   \, \rmd x\\
		&\leq (1 / (c \lambda_{k-1})  )^{  2\eta_k - 1  } \\
		&\leq   ( 1 / (c \lambda_{n} ) ) ^{  2\eta_k - 1  } \\
		&\leq   (c \lambda_{n}  )^{-1} \eqsp.
	\end{align*}
	It follows that
	\begin{align*}
		\sum_{k=1} ^n \mathbb E[Z_{k} ^{(2)2} |\mathcal F _{k-1}] \leq  n (c\lambda_n)^{-1}.
	\end{align*}
	Consequently, and using \eqref{eq:bound_unif_z2}, we can apply \Cref{th:freedman} with $ m = (c \lambda_{n}  )^{-1}$, $v = n (c \lambda_{n}  )^{-1}$ and we get that for all $n$ large enough such that $\log(n) / (n\lambda_n) \leq c$; this is made possible by \ref{cond:sequences}; and all $\gamma \geq 9$,\ak{dur à comprendre}
	\begin{align*}
		\mathbb P\left(\Big|\sum_{k=1}^n Z_{k} ^{(2)}   \Big|\geq  \sqrt {\gamma  (c\lambda_n)^{-1}  n\log(n)} \right) &\leq 2\exp\left(-\frac{\gamma  n\log(n)  }{2(  n  + \sqrt {\gamma   (c\lambda_n)^{-1} n\log(n)}  /3)} \right)\\
		&\leq 2\exp\left(-\frac{\gamma}{ 2( 1 +   \sqrt {\gamma  }  /3)}  \log(n)   \right)\\
		&= 2\exp\left(-\frac{\sqrt{\gamma}}{2} \frac{\sqrt{\gamma}}{1 +   \sqrt {\gamma  }  /3}  \log(n)   \right)\\
		&\leq 2\exp\left(-\frac{9}{4} \log(n) \right)\\
		& \leq 2\exp\left(-  2  \log(n)    \right) = 2 n^{-2},
	\end{align*}
	which series is convergent. We obtain the desired result by invoking the Borel-Cantelli lemma.
	
\end{proof}

\begin{prop}\label{prop:intial_bound_z1}
	Under \ref{cond:sequences}, \ref{cond:lowerbound}, \ref{cond:reg_f} and \ref{cond:kernel}, we have that for any $r>0$,
	\begin{align*}
		\sup_{\|x\|\leq n^r } \left|n^{-1}  \sum_{k=1} ^n   Z_{k} ^{(1)}(x) \right| = O \left( \sqrt { \frac{\log(n)  } { n h_n^d \lambda_n} } \right) \eqsp.
	\end{align*}
\end{prop}
\begin{proof}
	Using similar arguments as in the beginning of the  proof of Proposition \ref{prop:intial_bound_z2}, we can assume that   $ (\eta_k)_{k\geq 1} $ is valued in $ (1/2,1]$.
	The proof consists in applying \Cref{cor:freedmanx} to $Y_k(x) = Z_{k}^{(1)}(x)$ that is
	\begin{align*}
		Y_k(x) = W_k^{\eta_k} K_{h_n}( x -X_k) - \mathbb E [ W_k^{\eta_k} K_{h_n}( x -X_k) | \mathcal F_{k-1}] \eqsp.
	\end{align*}
	In the next few lines, we derive the quantities $m$, $v$, $\tau$ that appear in \Cref{cor:freedmanx}. Under \Cref{cond:kernel} combined with \Cref{prop:unif_bound_weights}, we have that
	\begin{align*}
		&|Y_k(x) | \leq \frac{K_\infty}{ c\lambda_n h_n^{d} } \eqsp.
	\end{align*}
	The previous bound corresponds to $m$ in \Cref{cor:freedmanx}. Moreover, using \Cref{prop:unif_bound_weights},
	\begin{align*}
		\mathbb E [ Y_k(x) ^2   |\mathcal F_{k-1}]  &\leq  \mathbb E [  W_k^{2\eta_k} K_{h_n}( x -X_k)^2 |\mathcal F_{k-1}] \\
		& = h_n^{-2d} \int \left(\frac{f(y) }{ q_{k-1} (y ) }\right) ^{2\eta_k -1 } f(y) K ( (x -y)/h_n )^2 \rmd y\\
		& \leq h_n^{-2d} (1 / (c\lambda_{k-1}) ) ^{2\eta_k -1 } \int  f(y) K ( (x -y)/h_n )^2 \rmd y\\
		& = h_n^{-d} (1 / (c\lambda_n) ) ^{2\eta_k -1 } \int  f(x-h_nu) K (u )^2 \rmd u\\
		&\leq h_n^{-d} ( c\lambda_n ) ^{ -1 } U K_\infty
	\end{align*}
	where we used a variable change $ u = (x -y)/h_n$ in the penultimate inequality, and the last inequality results from \Cref{cond:kernel}, $0< 2\eta_k - 1 \leq 1$, and \Cref{cond:reg_f}.
	%

	Hence, we get
	\begin{align*}
		\sum_{k=1} ^n   \mathbb E [ Y_k(x) ^2   |\mathcal F_{k-1}]\leq    n h_n^{-d} ( c\lambda_n ) ^{ -1 } U K_\infty \eqsp.
	\end{align*}
	The previous bound corresponds to $v$ in \Cref{cor:freedmanx}.
	Under \ref{cond:kernel}, $|  K_h( x +y  -X_k) - K_h( x -X_k)  | \leq L_K \| y/ h \| h^{-d}$ and it holds that for all $\|y\|\leq \epsilon$,
	\begin{align*}
		|Y_k(x+y ) - Y_k(x) | & \leq  ( W_k^{\eta_k} + \mathbb E [ W_k^{\eta_k} | \mathcal F_{k-1}] )   L_K \epsilon h_n^{-d-1}  .
	\end{align*}
	
	The l.h.s. of the previous inequality corresponds to $\tilde Y_k(x)$ in \Cref{cor:freedmanx}. We have
	\begin{align*}
		\sum_{k=1} ^ n \mathbb E [ \tilde Y_k(x) |  \mathcal F_{k-1}] \leq 2 n L_K \epsilon h_n^{-d-1}  .
	\end{align*}
	where we have used that $ \mathbb E [ W_k^{\eta_k} | \mathcal F_{k-1}]  \leq 1$.
	
	Taking $\epsilon =  h_n^{d+1} / n $, the value for $\tau$  in \Cref{cor:freedmanx} is $ 2L_K$. Let us now summarize the different factors taken to apply \Cref{cor:freedmanx}:
	\begin{align*}
		&m =\frac{K_\infty}{ c\lambda_n  h_n^{d} } \\
		&v =n h_n^{-d} ( c\lambda_n ) ^{ -1 } U K_\infty\\
		&\tau = 2L_K\\
		&\tilde v=\max(v,2m\tau) \leq  C\max( n / (\lambda_n h_n^d), \lambda_ n ^{-1} h_n^{-d} ) = C n /  (\lambda_n h_n^d)
	\end{align*}
	where $C$ is a positive constant.
	Let $\gamma >1$. We have, taking $t = \sqrt { \gamma n  \log(n) / ( h_n^d\lambda_n )  }$, $A=n^{r}$ and $\Omega_1 = \Omega$, for $n$ large enough $(t\geq \tau$ and $\tilde v \sqrt \gamma \geq 2mt/3$; this is made possible by \ref{cond:sequences}),
	\begin{align*}
		\mathbb P\Big(\sup_{\|x\|\le n^{r} }\big|\sum_{k=1} ^ n Y_k(x)\big|> 2 t\Big)&\leq  \mathbb P\Big(\sup_{\|x\|\le n^{r} }\big|\sum_{k=1} ^ n Y_k(x)\big|>t+\tau\Big)\\
		&\leq 4(1+2n^{r+1}/ h_n^{d+1} )^d \exp\left(-\frac{ t^2  }{8    (1+\sqrt \gamma) \tilde v} \right)\\
		&\leq 4(1+2n^{r+1}/ h_n^{d+1} )^d \exp\left(-\frac{\gamma   \log(n)  }{16 C     } \right)
	\end{align*}
	The last inequality holds because $\gamma >1$. It remains to choose $\gamma $ large enough in order to ensure the summability condition in the Borel Cantelli lemma.
	\end{proof}

\subsection{End of the proof of \Cref{prop:initial_bound}}

Since $f_n(x) = N_n(x) / D_n$ for all $n \geq 1$ and all $x \in \rset^d$, it is enough to show that $|D_n  - 1| =   o(1)$, and $\sup_{\|x\| \leq n^{r} } |N_n(x)   -  f(x) | =  o(1)$ . Both results are obtained independently, starting with $D_n$.

\textbf{Proof for $D_n$.} First note that for all $n \geq 1$, we can write the following decomposition for $D_n$
\begin{align*}
	D_n = n^{-1} \sum_{k=1} ^{n}   \int \tilde f_k    + n^{-1} \sum_{k=1} ^n   Z_{k} ^{(2)} \eqsp.
\end{align*}
Furthermore, under \ref{cond:sequences} and \ref{cond:lowerbound}, \Cref{prop:intial_bound_z2} implies that the second term of the r.h.s. is $O ( \sqrt { {\log(n)  } /  (n  \lambda_n) } )$ almost surely. Hence, by \ref{cond:sequences} the sequence $(D_n)_{n \geq 1}$ converges to $1$ as soon as $(n^{-1} \sum_{k=1} ^n  \int \tilde f_k  )_{n \geq 1}$ does. By the Cesaro lemma, this will be a consequence of having proven that $(\int \tilde f_k)_{k \geq 1}$ goes to $1$, which is what we set out to do next. For all $k \geq 1$, Jensen's inequality yields
$\int \tilde f_k   \leq  1 \eqsp$
and setting $\lambda_0 = 1$, we deduce using \ref{cond:lowerbound} that
\begin{align*}
	\int \tilde f_k = \int f^{\eta_k}  q_{k-1} ^{1-\eta_k}\geq  \lambda_{k-1}^{1-\eta_k}  \int f^{\eta_k}  q_0 ^{1-\eta_k} \geq  (\lambda_{k-1}c)^{1-\eta_k}  \eqsp.
\end{align*}
Thus, $(\int \tilde f_k)_{k \geq 1}$ goes to $1$ under \ref{cond:sequences} and we can conclude that $(D_n)_{n \geq 1}$ converges to $1$.

\textbf{Proof for $N_n$.}  For the numerator $N_n$, we follow a similar approach except that we need to deal with some convolution operator. For all $n \geq 1$ and all $x \in \rset^d$, we can write
\begin{align*}
	N_n(x)  = n^{-1} \sum_{k=1} ^{n}  (\tilde f_k\star K_{h_n}) (x)    + n^{-1} \sum_{k=1} ^n   Z_{k} ^{(1)},
\end{align*}
where given $r > 0$, the second term of the r.h.s is $O (\sqrt { \log(n)/ (n h_n^d \lambda_n) })$ as a consequence of \Cref{prop:intial_bound_z1}.

To treat the first term, use \ref{cond:lowerbound} and \ref{cond:kernel}, to obtain that for all $x \in \rset^d$,
\begin{align*}
	\lambda _ k  c f (x)  \leq  q_k(x) \leq  h_k^{-d}  K_\infty + U_{q_0}  : =  h_k^{-d} C,
\end{align*}
for some $C>0$. It follows that
\begin{align*}
	\left( n^{-1} \sum_{k=1} ^{n} f_k^{-}   \right)  \star K_{h_n}  (x)    \leq   n^{-1} \sum_{k=1} ^{n}  (\tilde f_k\star K_{h_n}) (x)  \leq \left( n^{-1} \sum_{k=1} ^{n} f^+_k \right)  \star K_{h_n} (x)
\end{align*}
with $f_k^{-} (x) =  f (x) (c\lambda_k) ^{1- \eta_{k}} $ and $f_k^+ (x) = f (x)^{\eta_k}  ( h_k^{-d} C ) ^{1- \eta_{k}}  $.

It remains to show that the previous lower and upper bounds converge to $f$\ak{under A3?} uniformly.  For   $1\leq p\leq +\infty$, let  $\|\cdot\|_p$ denote the $L_p(\lambda)$-norm. Using that $\|g\star \tilde g\|_\infty \leq  \|g\|_\infty \|\tilde g\|_1$, we find, for all collection of bounded functions $(g_1 , \ldots,  g_n)$,
\begin{align*}
	| \left( n^{-1} \sum_{k=1} ^{n} g_k   \right)  \star K_{h_n}  (x)    - f (x)| &\leq |\left( n^{-1} \sum_{k=1} ^{n} (g_k - f)    \right)  \star K_{h_n}  (x)  |   +  |f\star K_{h_n}  (x)  - f(x)|\\
	&\leq  \| n^{-1} \sum_{k=1} ^{n}  ( g_k- f) \|_\infty \|K_{h_n}\|_1   +  \|f\star K_{h_n}   - f \|_\infty\\
	&\leq  n^{-1} \sum_{k=1} ^{n}  \| g_k - f\|_\infty    +  \|f\star K_{h_n}   - f \|_\infty
\end{align*}
Hence, in virtue of the Cesaro lemma, the fact that $f_k^{-}$ and $f_k^{+}  $ both converge uniformly to $f$ enables to conclude that the first term in the latter upper bound goes to $0$. The fact that $\|f\star K_{h_n}   - f \|_\infty $ goes to $0$ is an easy consequence of \ref{cond:kernel}.

\qed

\begin{rem}\label{rk:bandwiths}
	Notice that in the latter proof, a different bandwith $h_k$ for each point $X_k$, $k=1,\dots,n$ could have been set (instead of $h_k = h_n$ for all $k$). Indeed, in the latter inequalities, the term  $ \|f\star K_{h_n}   - f \|_\infty$ would be replaced by $1/n \sum_{k=1}^n  \|f\star K_{h_k}   - f \|_\infty$, which also goes to $0$ by the Cesaro Lemma as soon as  $\|f\star K_{h_n}   - f \|_\infty $ goes to $0$ .
\end{rem}

\section{Proof of Proposition \ref{prop:initial_bound2}}\label{sec:proof_initial_bound2}

For the sake of completeness, we first recall the Inequality Reversal lemma as written in \cite[Lemma 1]{peel2010empirical}.
\begin{lemma}[Inequality Reversal lemma]\label{lemma:ineq:reversal} Let $X$ be a random variable and let $a, b > 0$, $c, d \geq 0$ be such that
	$$
	\forall t > 0 \eqsp, \quad \PP \lr{X \geq t} \leq a \exp \lr{- \frac{b t^2}{c + d t}} \eqsp.
	$$
	Then, with probability at least $1 - \delta$,
	$$
	|X| \leq \sqrt{\frac{c}{b} \ln \frac{a}{\delta}} + \frac{d}{b} \ln \frac{a}{\delta} \eqsp.
	$$
\end{lemma}

The proof  of \Cref{prop:initial_bound2} is an easy consequence of the following Lemma.

\begin{lemma}
	Under \ref{cond:sequences}, \ref{cond:lowerbound}, \ref{cond:tail_f} and \ref{cond:tail}, there exists $s_0\in \mathbb N$ large enough such that
	\begin{align}
		&\sup_{\|x\|> n^{s_0} } f_ n (x)  = o(1) \eqsp, \eqsp a.s \eqsp, \label{eq:unrest1}\\
		&\sup_{\|x\|> n^{s_0} } f(x)  = o(1)  \eqsp. \label{eq:unrest2}
	\end{align}
\end{lemma}

\begin{proof}
	We start with \eqref{eq:unrest1}. Let $n \geq 1$ and set $A = n^{s_0} / 2$ with $s_0 \in \nset$. For all $x \in \rset^d$, we have the following decomposition
	\begin{align}\label{eq:2term}
		f_n(x) = \sum_{k=1} ^n W_{n,k}^{(\eta_k)} K_{h_n} (x-X_k) \mathbb I_{\{\|X_k\| \leq A \}} +  \sum_{k=1} ^n W_{n,k}^{(\eta_k)} K_{h_n} (x-X_k)  \mathbb I_{\{\|X_k\|> A \}}  \eqsp.
	\end{align}
	Our goal is to prove that for $s_0$ large enough, both terms on the r.h.s of \eqref{eq:2term} go to $0$. We start by studying the first term of the r.h.s.
	
	(i) Proof for the first term of the r.h.s of \eqref{eq:2term}.  For any $ \|x\| > n^{s_0}$, we can write
	\begin{align*}
		\sum_{k=1} ^n W_{n,k}^{(\eta_k)} K_{h_n} (x-X_k) \mathbb I_{\{\|X_k\| \leq A \}} &\leq \sup_{1\leq k\leq n} \sup_{\|y\| \leq A} K_{h_n} (x-y) \\
		& \leq C_K h_n ^{-d}   \sup_{\|y\| \leq A,\, \|x\|>n^{s_0}  } (1+\|x-y\|/h_n) ^{-r_K}  \\
		&\leq C_K h _n ^{-d}   \sup_{\|y\| \leq A,\, \|x\|>n^{s_0}  } (1+\|x-y\|/h_1) ^{-r_K} \\
		&\leq C_K h _n ^{-d}  (1+ n^{s_0} / (2h_1)  ) ^{-r_K},
	\end{align*}
	where the last inequality follows from the fact that for all $x,y \in \rset^d$, $|\|x\|- \|y\|| \leq \|x-y\|$. We can then ensure that the previous term goes to $0$ by letting $s_0$ be large enough.
	
	(ii) Proof for the second term of the r.h.s of \eqref{eq:2term}. For the second term of the r.h.s, we have
	\begin{align}\label{eq:2term:lemma:unres}
		\sum_{k=1} ^n W_{n,k}^{(\eta_k)} K_{h_n} (x-X_k) \mathbb I_{\{\|X_k\|> A \}}
		&= \left( \sum_{k=1} ^n W_{k}^{\eta_k}  \right)^{-1}  \sum_{k=1} ^n W_{k}^{\eta_k}  K_{h_n} (x-X_k)  \mathbb I_{\{\|X_k\|> A \}} \nonumber\\
		&\leq \left( \sum_{k=1} ^n W_{k}^{\eta_k}  \right)^{-1} K_\infty h_n^{-d}   \sum_{k=1} ^n  W_{k}^{\eta_k}   \mathbb I_{\{\|X_k\|> A \}} \eqsp,
	\end{align}
	and we are thus interested in studying the r.h.s of \eqref{eq:2term:lemma:unres}. A first remark is that using \Cref{prop:intial_bound_z2}, we obtain that almost surely
	\begin{align}\label{eq:3}
		\left( \sum_{k=1} ^n W_{k}^{\eta_k}  \right) ^{-1} = n^{-1}  (1 + o(1) ) ^{-1} \eqsp .
	\end{align}
	We now move on to the study of $\sum_{k=1} ^n  W_{k}^{\eta_k}   \mathbb I_{\{\|X_k\|> A \}}$ in \eqref{eq:2term:lemma:unres}. To do so, for all $k \geq 1$, let us define $p_k(A) =  \mathbb E [ W_{k}^{\eta_k}   \mathbb I_{\{\|X_k\|> A \}} | \mathcal F_{k-1} ] $ and $Z_k^{(3)}(A) = W_{k}^{\eta_k}   \mathbb I_{\{\|X_k\|> A \}} - p_k(A)$ so that
	\begin{align*}
		\sum_{k=1} ^n  W_{k}^{\eta_k}   \mathbb I_{\{\|X_k\|> A \}}  =   \sum_{k=1} ^n Z_k^{(3)}(A) + \sum_{k=1} ^n p_k (A) \eqsp.
	\end{align*}
	Then, for all $k \geq 1$, it holds that
	\begin{align*}
		p_k(A) &= \int_{\|x\|>A} f(x)^{\eta_k} q_{{k-1}} (x) ^{1-\eta_k} \, \rmd x\\
		&\leq (c\lambda_{{k-1}} ) ^{1-\eta_k} \int_{\|x\|>A} f(x) \, \rmd x \\
		&=  (c\lambda_{{k-1}} ) ^{1-\eta_k} p(A) \eqsp,
	\end{align*}
	with $p(A) \eqdef  \int _{\|x\| > A} f(x) \, \rmd x $. Using \ref{cond:sequences}, we deduce that there exists a constant $C>0$ such that
	\begin{align*}
		\sum_{k=1} ^n p_k (A)  \leq    C n p(A) \eqsp.
	\end{align*}
	Furthermore, under \ref{cond:tail_f}, Markov's inequality yields
	\begin{align}\label{eq:2term:markov}
		p(A) \leq A^{-\delta} \int \|x\|^\delta f(x) \, \rmd x
	\end{align}
	and as a consequence, we obtain
	\begin{align}\label{eq:1}
		\sum_{k=1} ^n p_k (A)  \leq    C n A^{-\delta} \int \|x\|^\delta f(x) \, \rmd x \eqsp.
	\end{align}
	Additionally, observe that $\sum_{k=1} ^n Z_k^{(3)}(A) $ is a sum of martingale increments so our next step will be to apply \Cref{th:freedman}. For this purpose, note that under \ref{cond:lowerbound} and \ref{cond:tail_f}, we can write for all $k = 1, \ldots , n$, 
	$$
	W_k \leq \lambda_{k-1}^{-1} C_0 (1 + \|X_k\|^\delta)^{-1} \leq \lambda_{n}  ^{-1} C_0 (1 + \|X_k\|^\delta)^{-1}
	$$
	so that
	\begin{align}\label{eq:unifboundZk}
	|Z_{k} ^{(3)}(A)| \leq \lambda_{n}  ^{-\eta_k} C_0^{\eta_k} \sup_{\|x\|>A}  (1+\|x\|  ^{\delta} ) ^{-\eta_k} \leq \lambda_{n}  ^{-1}  C_0 (1 +  A) ^{-\delta}    \eqsp.
	\end{align}
	We now treat the two case $k\geq k_0$ and $k < k_0$ separately.

\begin{itemize}
  \item When $k\geq k_0$ (such that $0< 2\eta_k - 1 \leq 1$), we can write
	\begin{align*}
		\mathbb E[Z_k ^{(3)}(A)^2  |\mathcal F _{k-1}] &\leq \mathbb E[ (f(X_k ) / q_{k-1} (X_k) )^{2\eta_k} \mathbb I_{\{\|X_k\|> A \}} |\mathcal F _{k-1}]\\
		& =  \int  _{\|x\| > A}   f(x )^{2\eta_k }  q_{k-1} (x)^{1 - 2\eta_k } \, \rmd x\\
		&= \int _{\|x\| > A}   ( f(x )/ q_{k-1} (x) )  ^{2\eta_k -1 }  f (x)   \, \rmd x\\
		&\leq (1 / (c \lambda_{k-1})  )^{  2\eta_k - 1  }   p(A) \\
		&\leq   (c \lambda_{n}  )^{-1}  p(A)\\
		&\leq (c \lambda_{n}  )^{-1} A^{-\delta} \int \|x\|^\delta f(x) \, \rmd x
		\eqsp.
	\end{align*}
where we have used \eqref{eq:2term:markov} in the last inequality.

  \item When $k< k_0$, we have $  \mathbb E[Z_k ^{(3)}(A)^2  |\mathcal F _{k-1}] \leq  M A^{-\delta}    $ for some constant $M>0$ that can be deduced from the almost sure bound \eqref{eq:unifboundZk} given just before.
\end{itemize}

It follows that, when $n $ is large enough, for all $k =1,\ldots, n$,
	\begin{align*}
		\mathbb E[Z_k ^{(3)}(A)^2  |\mathcal F _{k-1}]  &\leq   (c \lambda_{n}  )^{-1} A^{-\delta} \int \|x\|^\delta f(x) \, \rmd x\eqsp.
	\end{align*}
	and therefore,
	\begin{align*}
		\sum_{k=1} ^n \mathbb E[Z_{k} ^{(3)}(A)^2 |\mathcal F _{k-1}] & \leq  n (c\lambda_n)^{-1} A^{-\delta} \int \|x\|^\delta f(x) \, \rmd x \eqsp.
	\end{align*}
	Consequently, we can apply \Cref{th:freedman} with $ m =C_0   \lambda_{n}  ^{-1} (1 +  A) ^{-\delta}   $, $v = n (c \lambda_{n}  )^{-1} A^{-\delta} \int \|x\|^\delta f(x) \, \rmd x $ and we obtain that for all $t > 0$
	\begin{align*}
		\mathbb P\Bigg(\Big|\sum_{k=1}^n  Z_{k} ^{(3)}(A)\Big|\geq t\Bigg) &\leq 2\exp\left(-\frac{t^2}{2({v}+t{m}/3)} \right) \eqsp.
	\end{align*}
	Inverting this inequality using \Cref{lemma:ineq:reversal}, we get that, with probability $1-1/n^{2}$,
	\begin{align}
		\nonumber \left|\sum_{k=1}^n Z_{k} ^{(3)}(A)   \right| &\leq \sqrt{  4v  \log(2 n) }  + (2m/3) \log( 2n)\\
		\label{eq:2}&=\sqrt{  4 n (c \lambda_{n}  )^{-1} A^{-\delta} \lr{ \int \|x\|^\delta f(x) \, \rmd x}  \log(2 n) }  + (2/3) C_0   \lambda_{n}  ^{-1} (1 +  A) ^{-\delta}    \log( 2n) .
	\end{align}
	Invoking the Borel-Cantelli lemma we obtain that the previous bound is an almost sure rate.

	Putting together \eqref{eq:3}, \eqref{eq:1} and \eqref{eq:2} in \eqref{eq:2term:lemma:unres}, we obtain the almost-sure bound
	\begin{align*}
		& \left| \sum_{k=1} ^n W_{n,k}^{(\eta_k)} K_{h_n} (x-X_k) \mathbb I_{\{\|X_k\|> A \}} \right|\\
		& = O \left( n ^{-1}  h_n^{-d}  \left(  \sqrt{  n  A^{-\delta} \lambda_{n}  ^{-1}   \log(2 n) }  +   (A^{-\delta} \lambda_{n} ) ^{-1}   \log( 2n) + nA^{-\delta}  \right)\right) \eqsp.
	\end{align*}
	We easily obtain that the previous bound goes to $0$ provided that $s_0$ is large enough, which concludes the proof of \eqref{eq:unrest1}.
	
	As for \eqref{eq:unrest2}, notice that by \ref{cond:tail_f}, for all $\| x \| > n^{s_0}$
	$$
	f(x) \leq \frac{C_0 q_0(x)}{1 + \|x \|^\delta} \leq \frac{C_0 q_0(x)}{1 + n^{s_0 \delta }},
	$$
	we obtain \eqref{eq:unrest2} when $s_0$ is large enough.
	

	

\end{proof}

\section{Proof of \Cref{prop:tuning_eta}}\label{sec:proof_tuning_eta}



\begin{proof} We start by proving (i) and (ii).

Proof of (i) and (ii). First note that for our choice of $\PP$ and $\PQ$, we can write
\begin{align*}
D_1(\PP||\PQ) &= \sum_{\ell =1}^{m_k} W_{k,\ell} \log \lr{ \frac{W_{k,\ell}}{1/m_k}} = \sum_{\ell =1}^{m_k} W_{k,\ell} \log \lr{W_{k,\ell}} + \log(m_k) \\
& \leq \sum_{\ell =1}^{m_k} W_{k,\ell} \lr{W_{k,\ell} -1} + \log(m_k)
\end{align*}
where we have used that $\log(x)\le x-1$ for $x>0$. Thus, we have that
\begin{align}\label{eq:logmk}
D_1(\PP||\PQ) \leq \log(m_k) \eqsp.
\end{align}
In addition, \cite[Theorem 3]{2012arXiv1206.2459V} implies that for all $\alpha \in [0,1]$
$$
0 \leq D_\alpha (\PP||\PQ) \leq D_1 (\PP||\PQ)
$$
where equality is reached if and only if $\PP = \PQ$. Now combining with \eqref{eq:logmk} and by definition of $\eta_{k, \alpha}$ in \eqref{eq:ada}, we deduce that for all $\alpha \in [0,1]$,
$$
0 \leq \eta_{k,1} \leq \eta_{k, \alpha} \leq 1 \eqsp,
$$
and that $\eta_{k, \alpha} = 1$ if and only if $\PP = \PQ$.
	%

Proof of (iii). We assume that $\lim_{k\to \infty} m_k = m$. Hence, for $k$ big enough, $m_k =m$.  Hence we will derive the proof as if $m_k=m$ for all $k.$
 A first remark is that thanks to (ii), it is enough to prove that $\lim_{k\to \infty}   \eta_{k,1}  = 1$  in $L_1$ to obtain that for all $\alpha \in [0,1]$, $\lim_{k\to \infty}   \eta_{k,\alpha}  = 1$ in $L_1$, that is $\lim_{k\to\infty}\mathbb{E} [|\eta_{k,\alpha} - 1|] = 0$.

Since that for all $k \geq 1$,
\begin{align*}
 \eta_{k,1} = 1 - \frac{D_1 (\PP||\PQ)}{\log(m)} = -\frac{\sum_{\ell =1}^{m} W_{k,\ell} \log \lr{W_{k,\ell}}}{\log(m)}
\end{align*}
the proof is concluded if we can prove that, in $L_1$
\begin{align}\label{eq:goalada}
\lim_{k \to \infty} \sum_{\ell =1}^{m} W_{k,\ell} \log \lr{W_{k,\ell}} = - \log(m) \eqsp.
\end{align}


To see this, let us define the two maps $g_1:(w_1,\ldots, w_m) \mapsto \sum_{\ell =1} ^m w_\ell \log(w_\ell )$ and $g_2:(w_1,\ldots, w_m) \mapsto \lr{\sum_{\ell' =1} ^m w_\ell'}^{-1} (w_1,\ldots, w_m)$. Observe then that the map $g_1$ is a continuous transformation (defined on the simplex) and $g_2$ is a continuous transformation on the space of nonnegative weights.

If we further denote by $(\tilde W_{k,\ell})_{\ell = 1}^{m}$ the unnormalised weights (i.e  $\tilde W_{k,\ell} = f(X_{k, \ell}) / q_{k - 1} (X_{k, \ell}) $ for all $\ell = 1 \ldots m$), then it is enough to show that $(\tilde W_{k,1},\ldots, \tilde W_{k,m})$ converges to $(1, \ldots, 1)$ in $L_1$ to prove \eqref{eq:goalada}. 
Since for all $\ell =1,\ldots, m$,
\begin{align*}
 \mathbb E [|  \tilde W_{k,\ell} - 1|   ]  = \int | f  - q_{k-1}   |  \eqsp,
\end{align*}
this follows from Scheffé's lemma and the proof is concluded. 

\end{proof}


\section{Additional Experiments}\label{sec:add_expes}

\begin{figure}[h]
	\centering
		\includegraphics[scale=0.27]{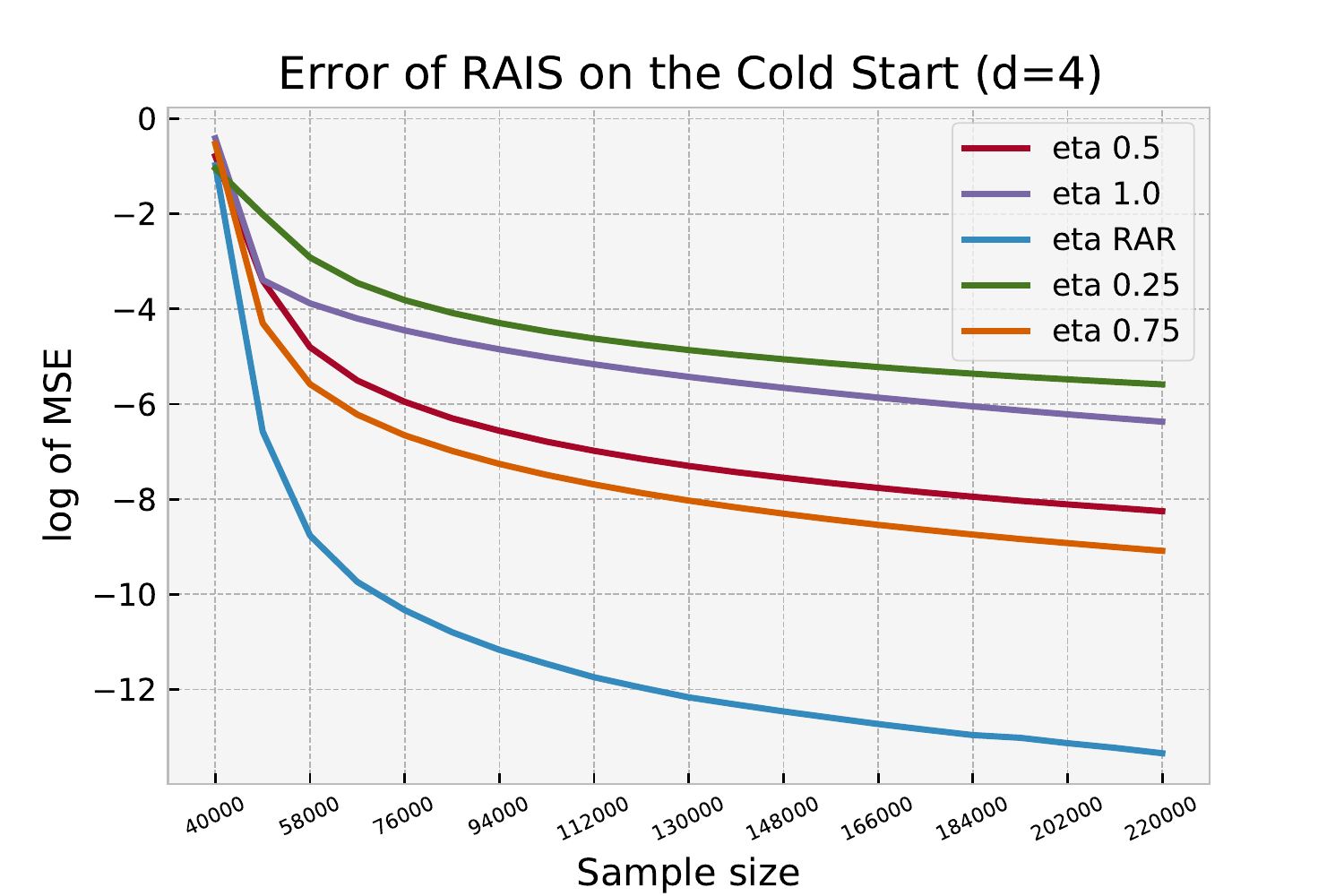}\includegraphics[scale=0.27]{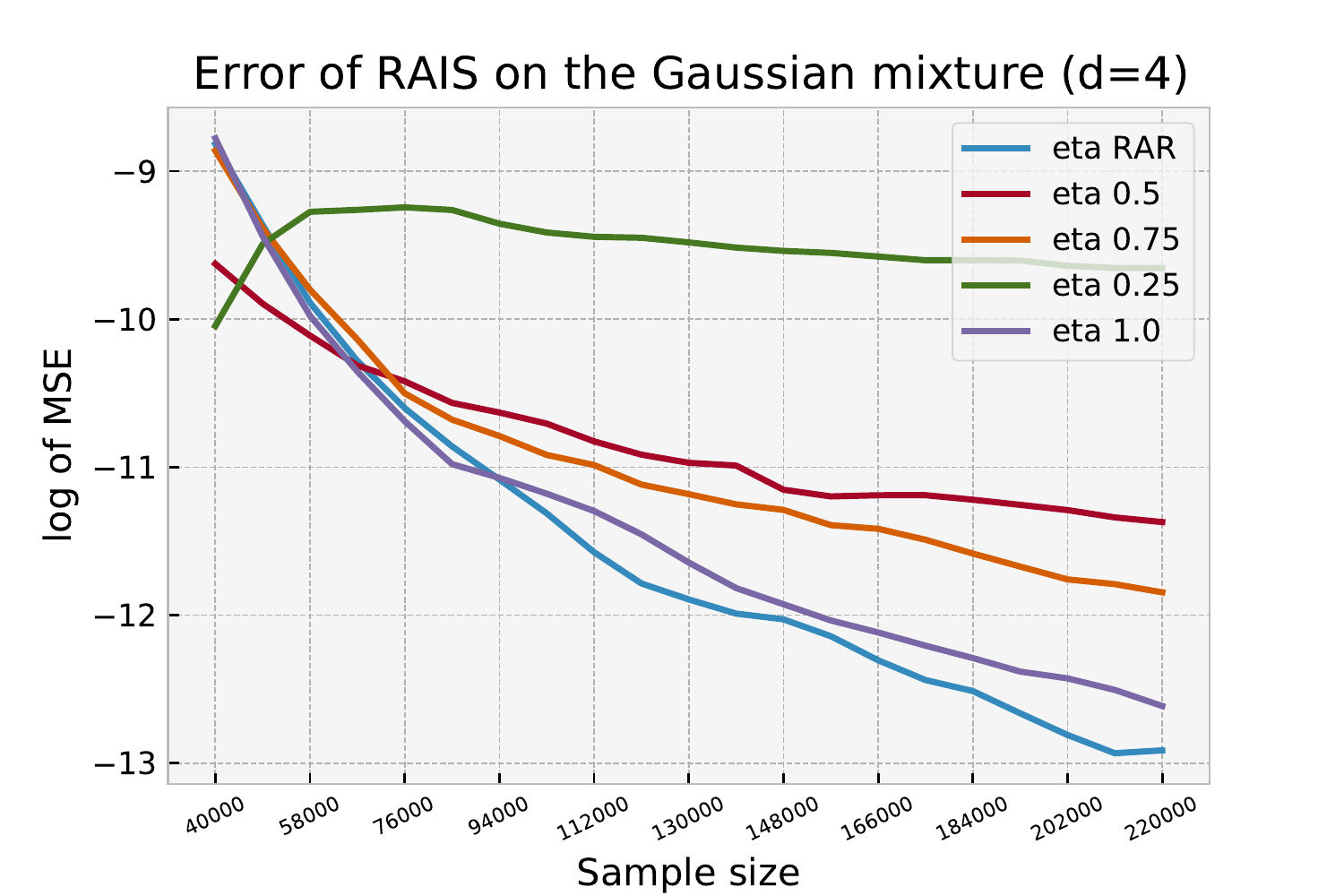}\includegraphics[scale=0.27]{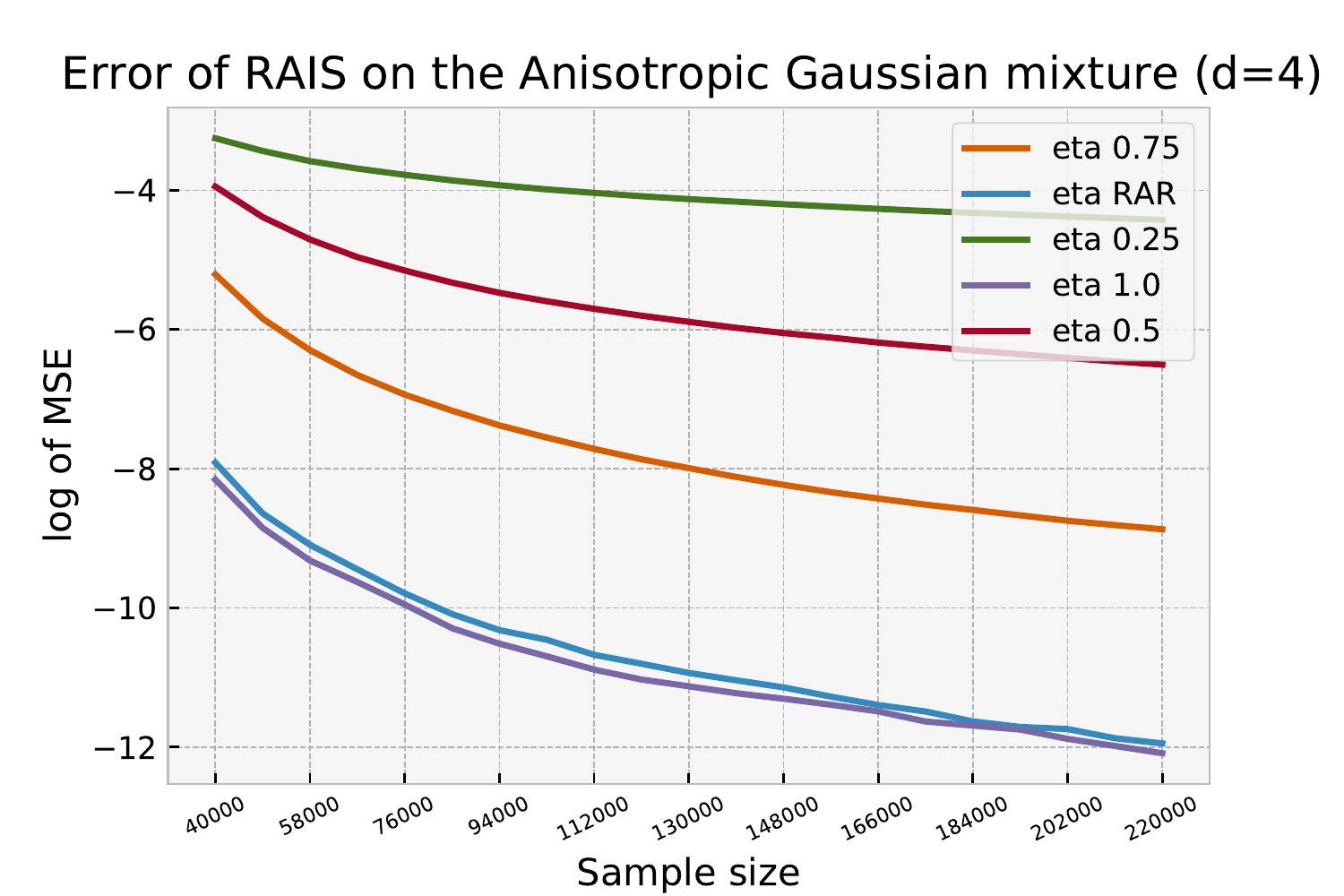} \\		
		\includegraphics[scale=0.28]{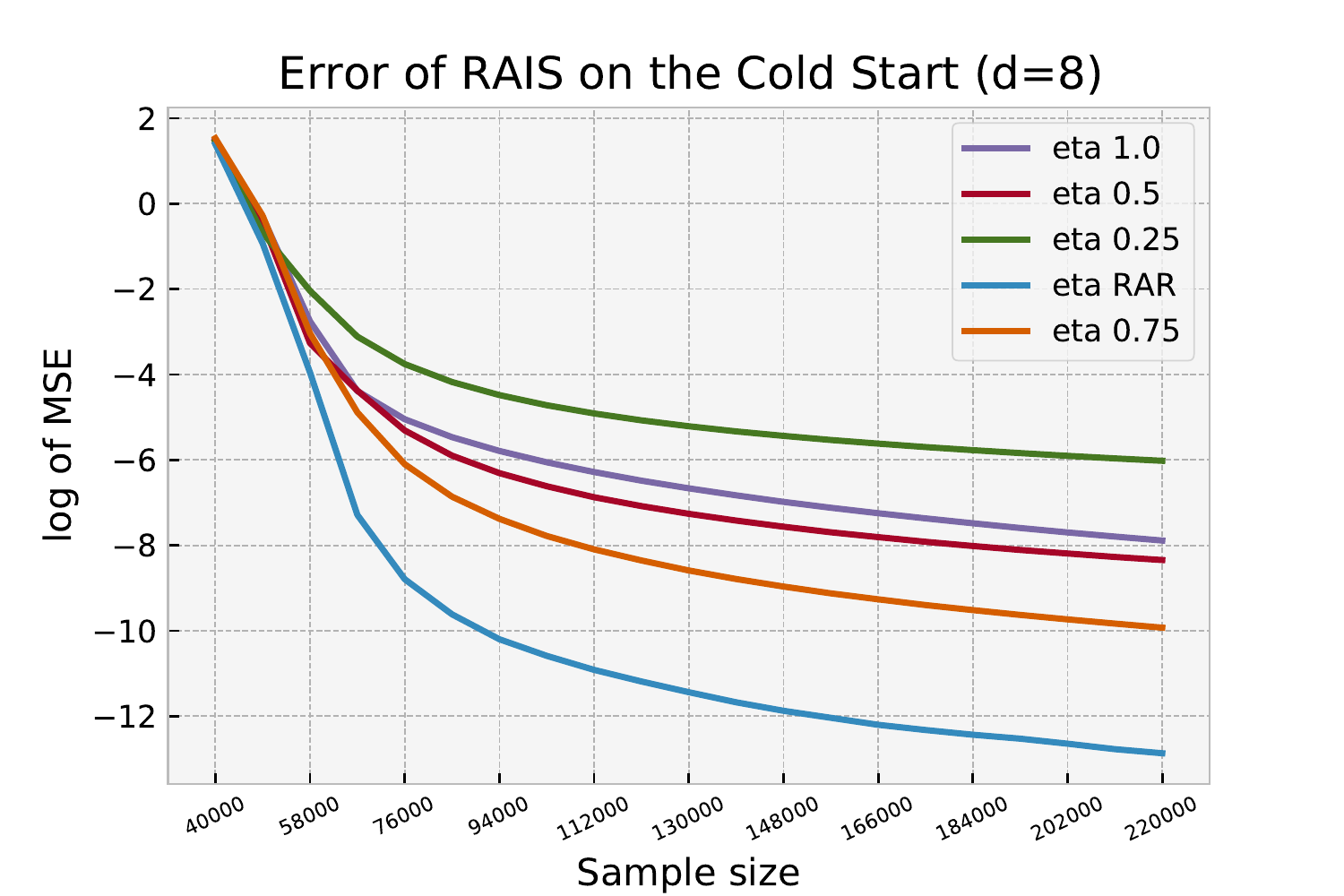}\includegraphics[scale=0.27]{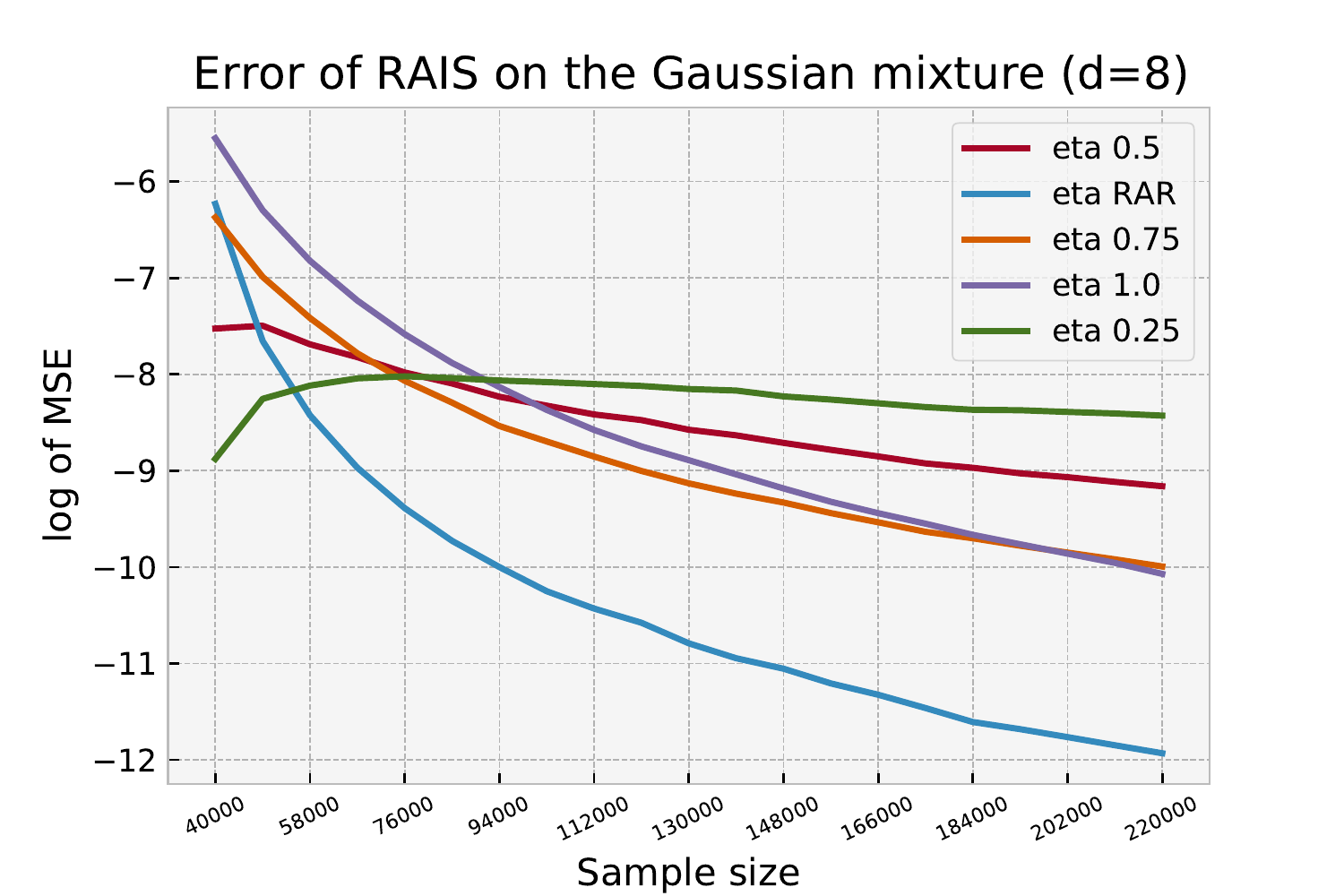}\includegraphics[scale=0.28]{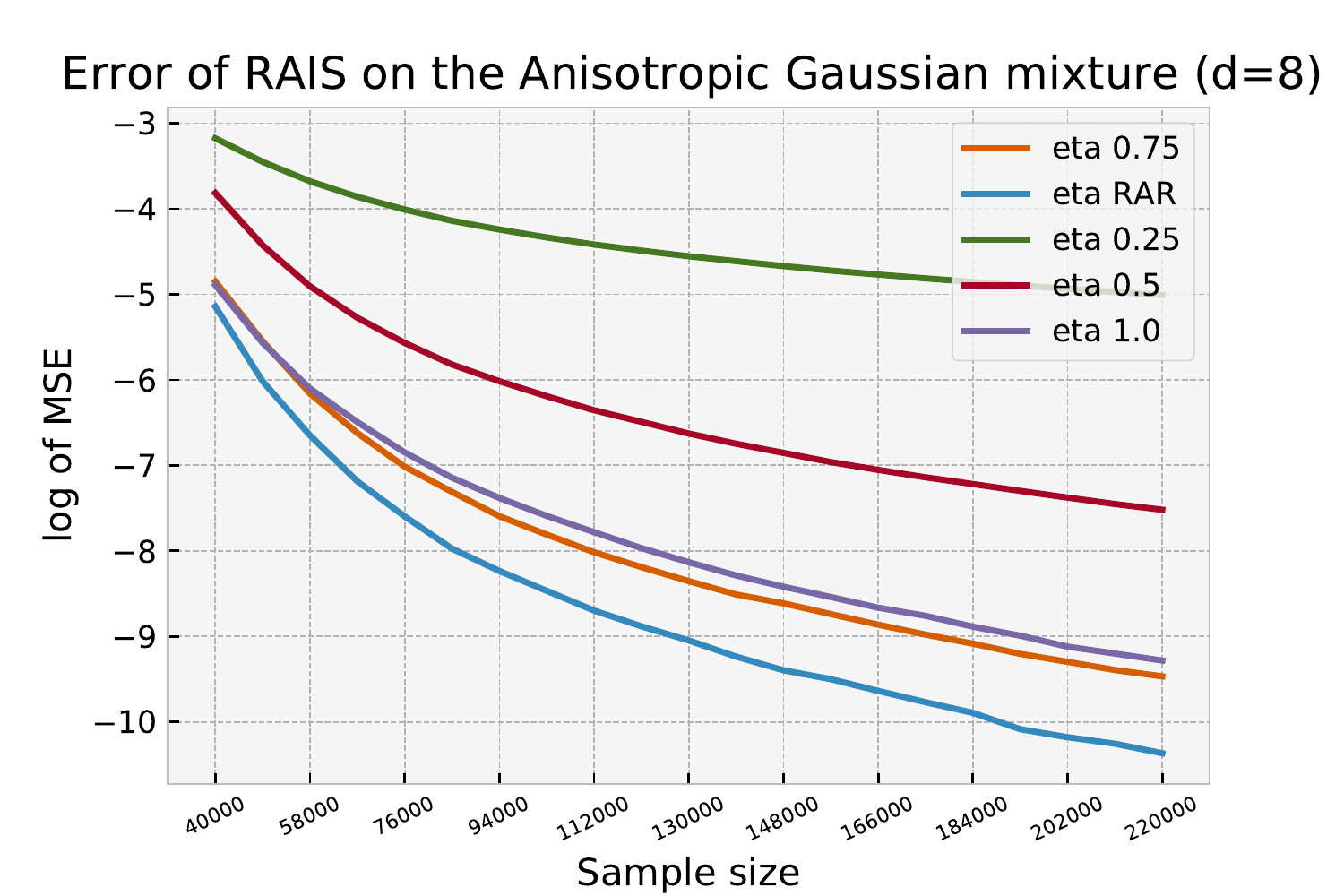}
	
	\caption{Logarithm of the average squared error, computed over 50 replicates. 
	}\label{fig:toy_error2}
\end{figure}

\begin{figure}[h]
	\centering
		\includegraphics[scale=0.28]{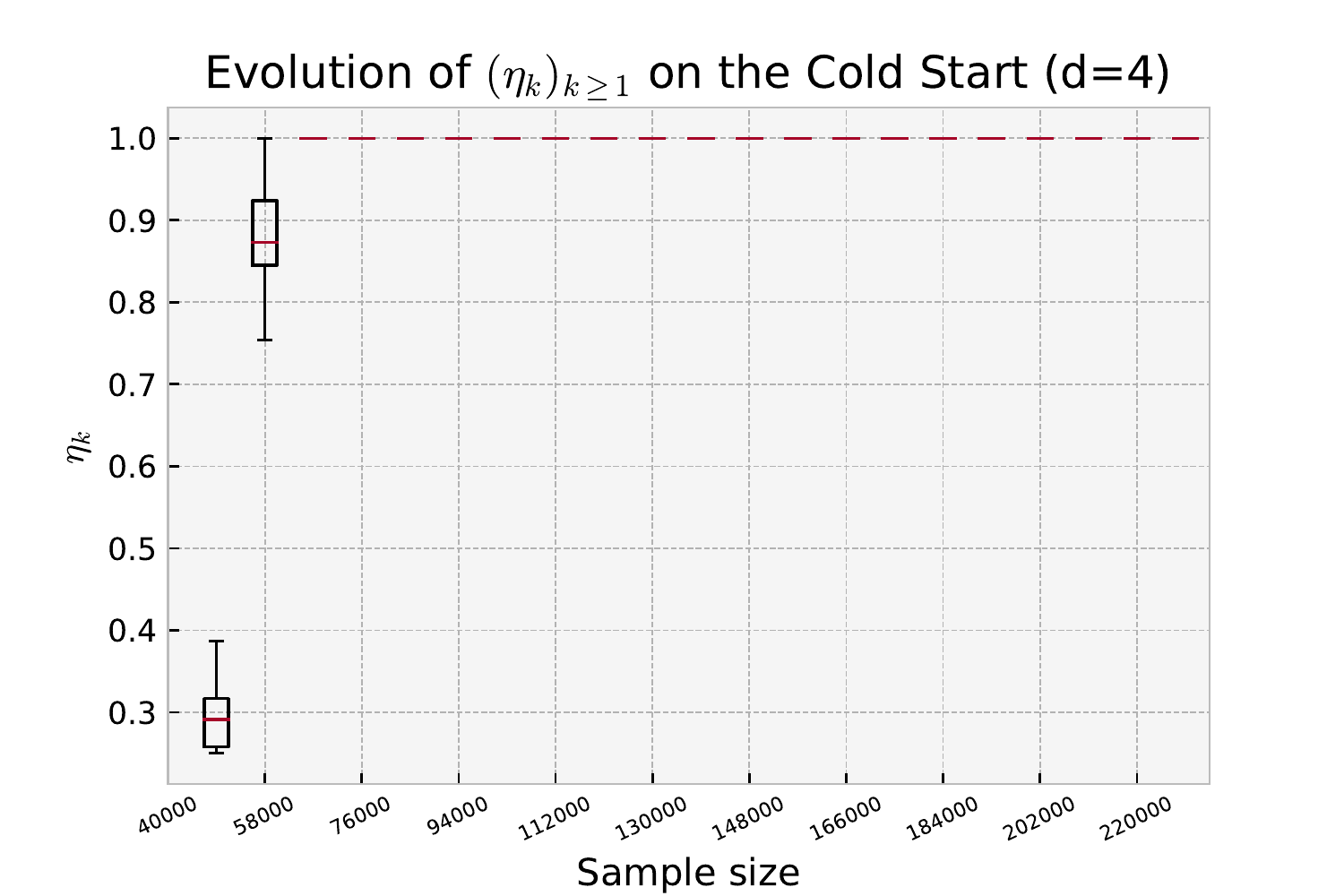}\includegraphics[scale=0.28]{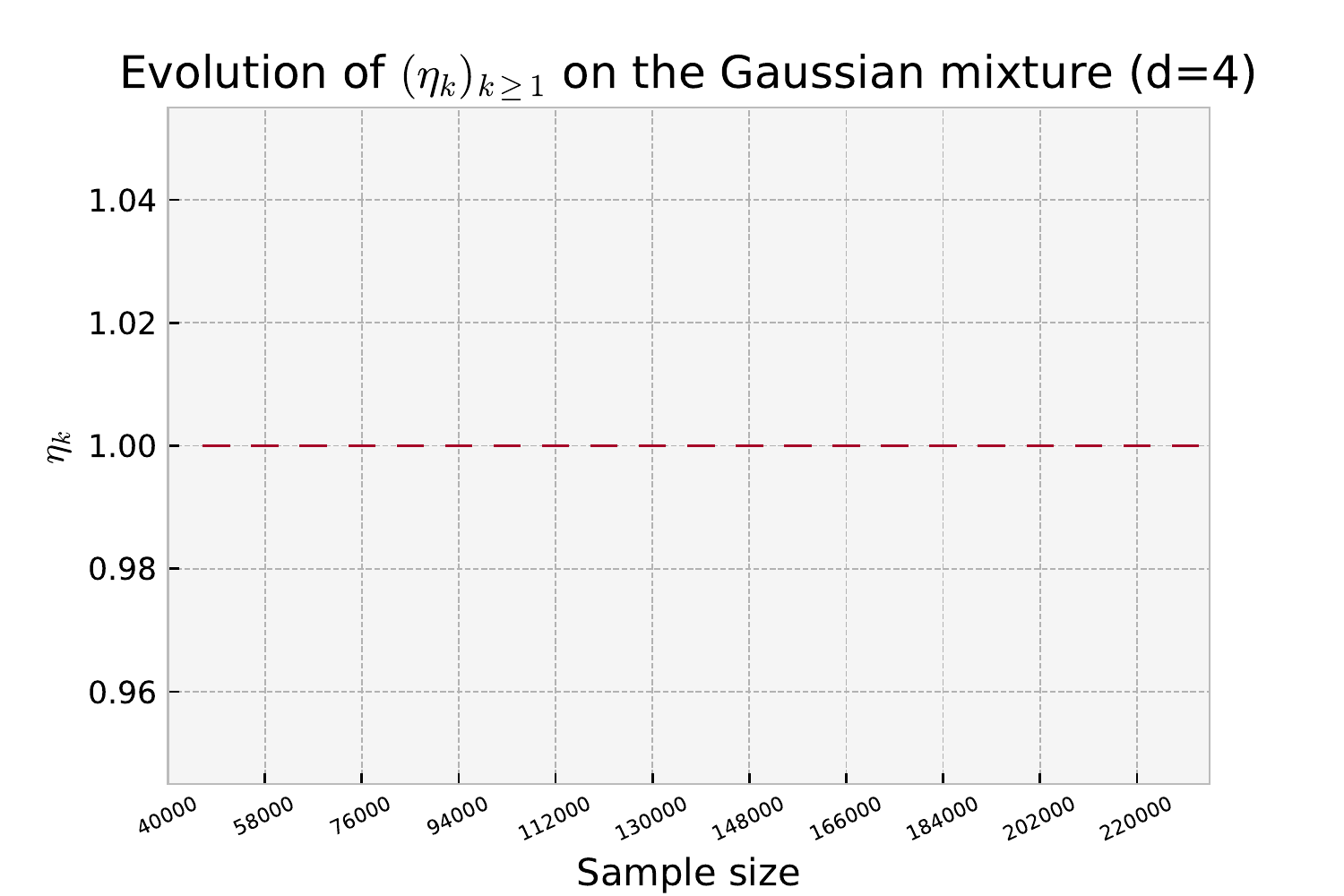}\includegraphics[scale=0.28]{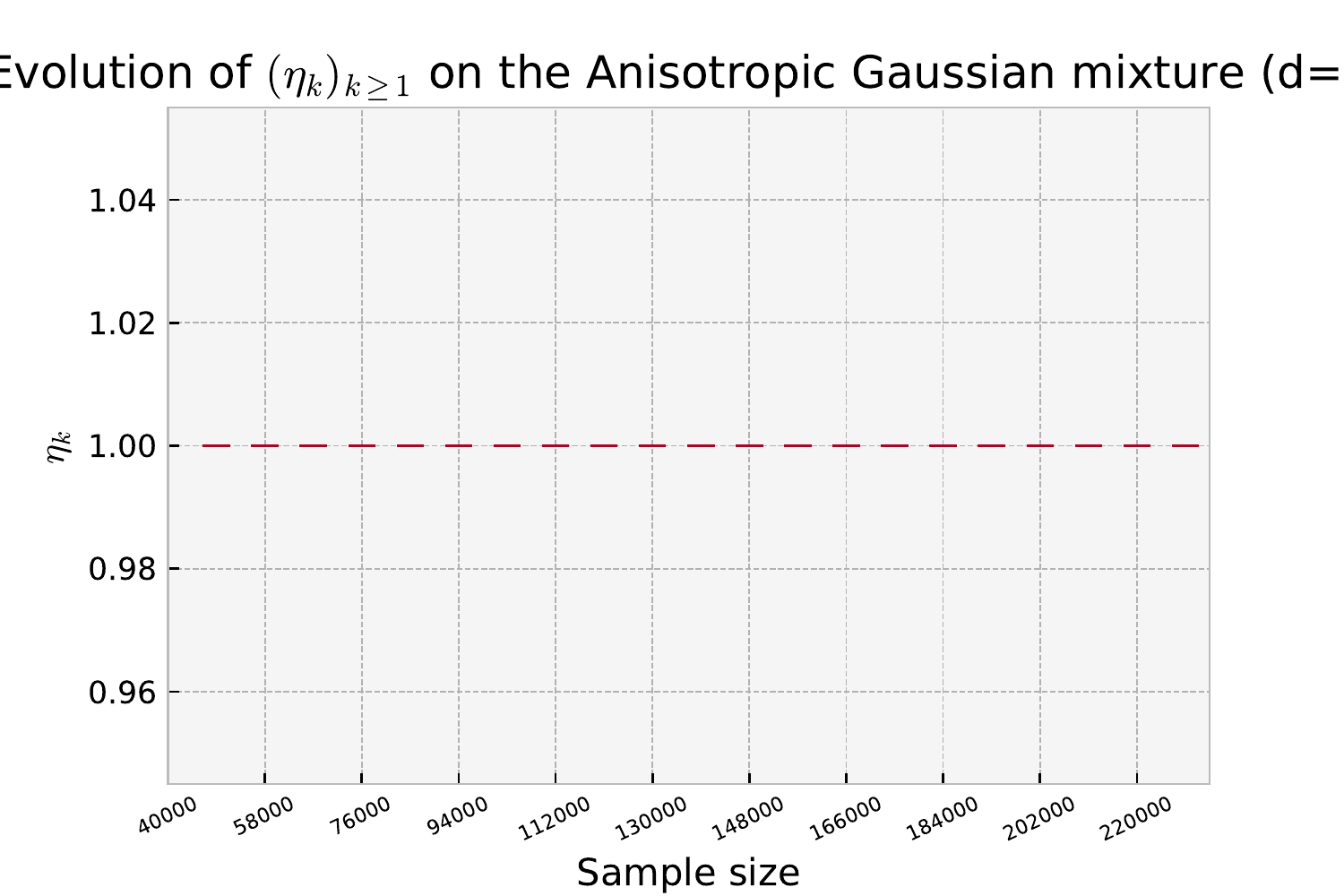}\\ 
		\includegraphics[scale=0.28]{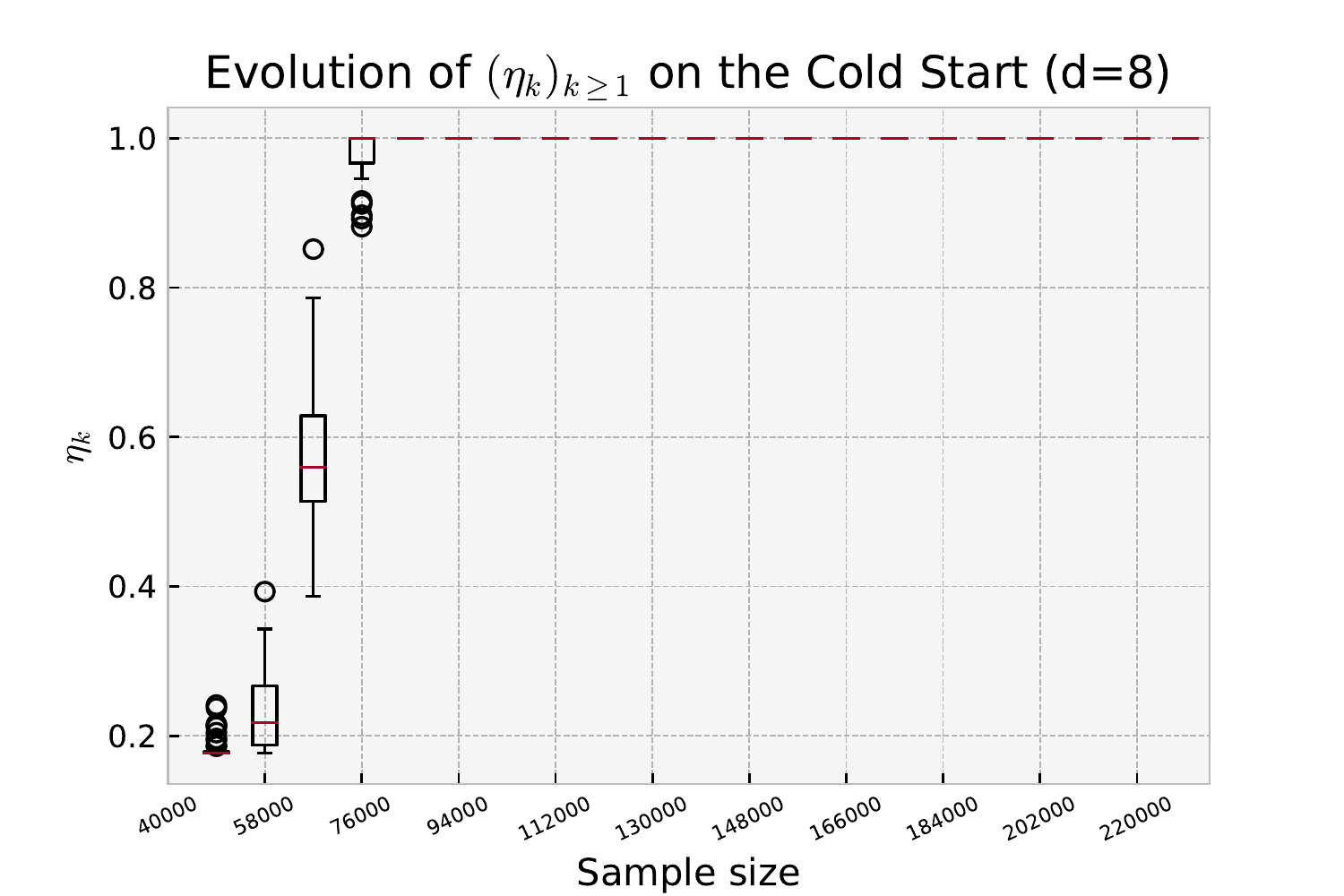}\includegraphics[scale=0.28]{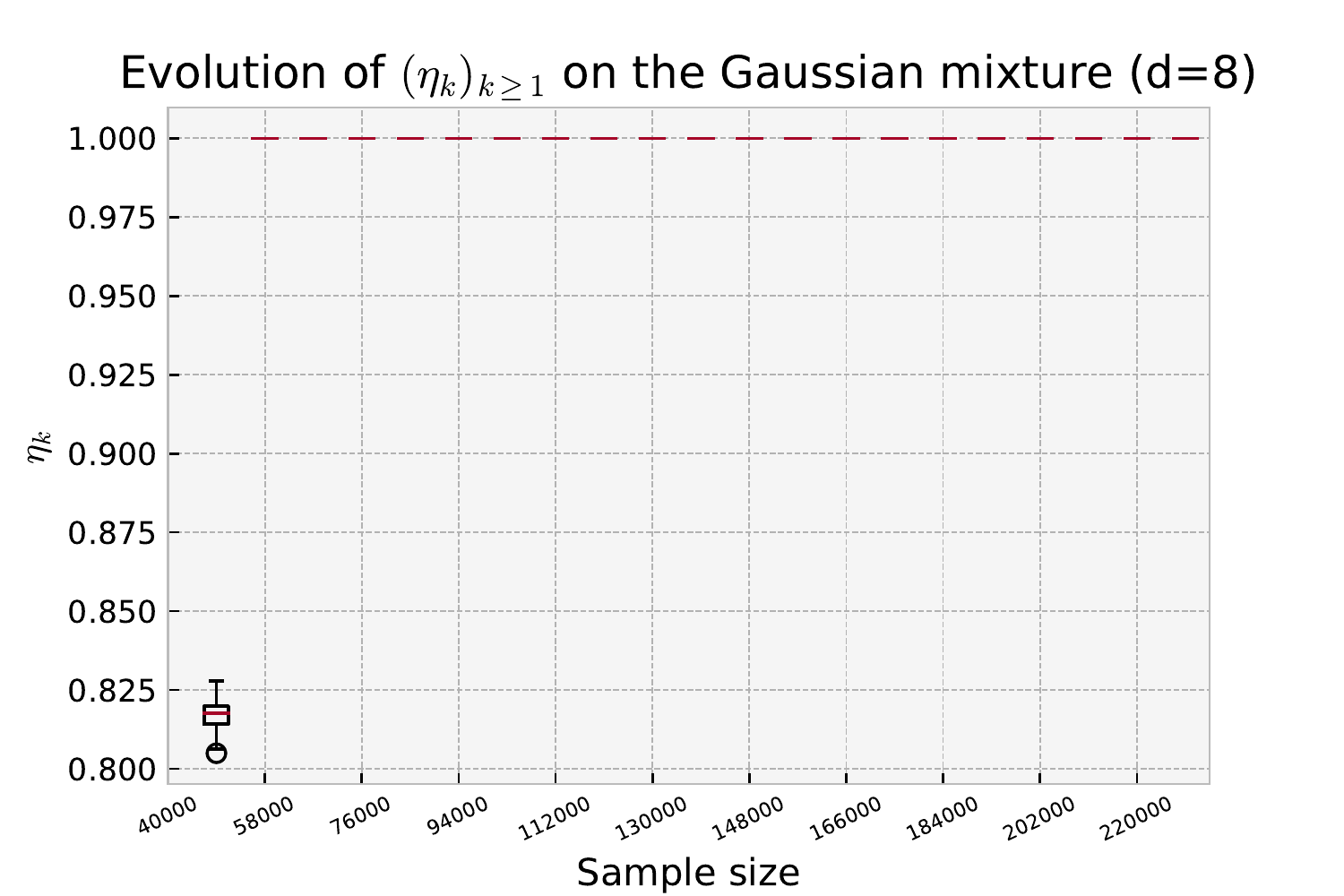}\includegraphics[scale=0.28]{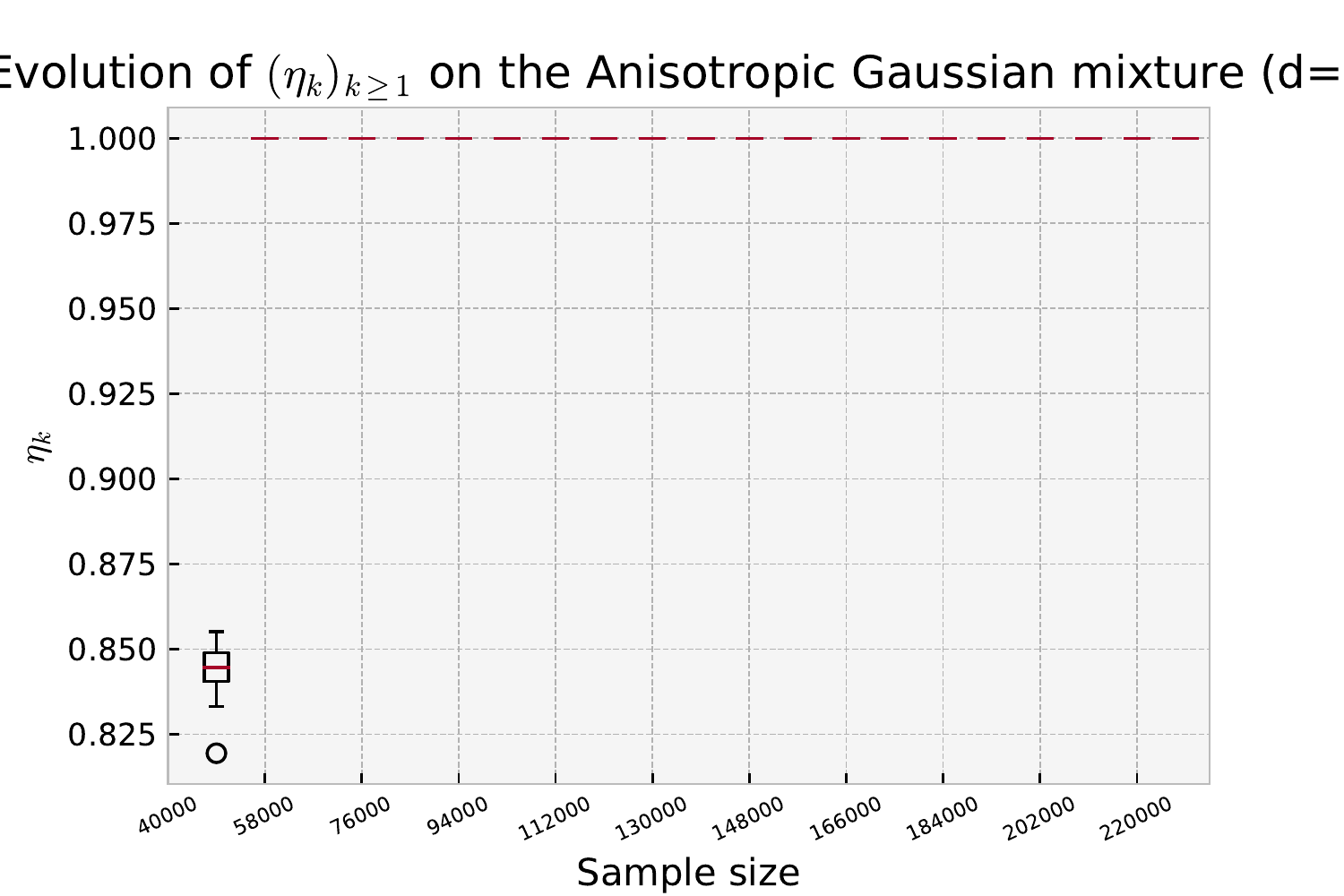} 
	
	\caption{Boxplot of the values of $\eta$ obtained from the ADA strategy.
	}\label{fig:toy_bp2}
\end{figure}

We report in \Cref{fig:toy_error2} the results of the proposed method, and in \Cref{fig:toy_bp2} the evolution of $\eta$ in smaller dimensions. We can see  in \Cref{fig:toy_error2} that \Cref{algo:RIS} along with the subroutine \Cref{algo:ADA} always outperform the competitive schedules for the regularization, and in \Cref{fig:toy_bp2} that  \Cref{algo:ADA} converges to 1.